\documentclass[11pt]{article}
\usepackage[dvips]{graphicx}
\usepackage{amssymb,amsmath,color}
\usepackage{url}

\usepackage[mathcal]{eucal}

\usepackage[colorlinks=true, allcolors=blue, pagebackref]{hyperref}       % hyperlinks
\renewcommand*{\backrefalt}[4]{%
    \ifcase #1 \footnotesize{(not cited)}%
    \or        \footnotesize{(cited on page~#2)}%
    \else      \footnotesize{(cited on pages~#2)}%
    \fi}

\DeclareMathAlphabet\mathbfcal{OMS}{cmsy}{b}{n}

\newcommand{\BEAS}{\begin{eqnarray*}}
\newcommand{\EEAS}{\end{eqnarray*}}
\newcommand{\BEA}{\begin{eqnarray}}
\newcommand{\EEA}{\end{eqnarray}}
\newcommand{\BEQ}{\begin{equation}}
\newcommand{\EEQ}{\end{equation}}
\newcommand{\BIT}{\begin{itemize}}
\newcommand{\EIT}{\end{itemize}}
\newcommand{\BNUM}{\begin{enumerate}}
\newcommand{\ENUM}{\end{enumerate}}
\newcommand{\BA}{\begin{array}}
\newcommand{\EA}{\end{array}}

\newcommand{\tr}{\mathop{ \rm tr}}

\newcommand{\idm}{I}
\newcommand{\rb}{\mathbb{R}}
\newcommand{\cb}{\mathbb{C}}
\newcommand{\zb}{\mathbb{Z}}

\newcommand{\ds}{\displaystyle }
\newcommand{\BlackBox}{\rule{1.5ex}{1.5ex}}  % end of proof

\newenvironment{proof}{\par\noindent{\bf Proof\ }}{\hfill\BlackBox\\[2mm]}
\newtheorem{lemma}{Lemma}
\newtheorem{theorem}{Theorem}

\newtheorem{corollary}{Corollary}

\newcommand{\mysec}[1]{Section~\ref{sec:#1}}
\newcommand{\eq}[1]{Eq.~(\ref{eq:#1})}

\title{Exponential convergence of sum-of-squares hierarchies  \\ for trigonometric polynomials}

\author{Francis Bach and Alessandro Rudi\\
Inria,  Ecole Normale Sup\'erieure \\
PSL Research University \\
\url{francis.bach@inria.fr}, \url{alessandro.rudi@inria.fr}}

\date{\today}

\begin{document}
\maketitle
 
 \begin{abstract}
     We consider the unconstrained optimization of multivariate trigonometric polynomials by the sum-of-squares hierarchy of lower bounds. We first show a convergence rate of $O(1/s^2)$ for the relaxation with degree $s$ without any assumption on the trigonometric polynomial to minimize. Second, when the polynomial has a finite number of global minimizers with invertible Hessians at these minimizers, we show an exponential convergence rate with explicit constants. Our results also apply to the minimization of regular multivariate polynomials on the hypercube.
 \end{abstract}

  \section{Introduction}
 
 Sum-of-squares hierarchies provide an elegant framework for global optimization for a variety of hard optimization problems. Starting from continuous polynomial optimization and combinatorial optimization problems~\cite{lasserre2001global,parrilo2003semidefinite}, they now apply to many other infinite-dimensional optimization problems such as optimal transport or optimal control (see a thorough review in~\cite{lasserre2010moments,henrion2020moment}).

 Within optimization, they are often cast as the minimization of multivariate polynomials over sets defined by essentially arbitrary polynomial constraints. They work by solving a sequence of semi-definite programming problems of increasing sizes, often referred to as a sum-of-squares (SOS) ``hierarchy'' of optimization problems.
 
 The convergence rate of the minimal values of these problems towards the optimal value is empirically much faster than can actually be shown. Current theoretical results can be summarized as follows:
 \BIT
 \item In dimension one, there is no need for hierarchies, as the most direct formulations are tight~\cite{nesterov2000squared}.
 \item In higher dimensions, {under mild assumptions}, the hierarchies are always converging, due to powerful representation results of strictly positive polynomials~\cite{putinar1993positive,schmudgen2017moment}. However, finite convergence can only be shown when strict second-order local optimality conditions are satisfied, but without a bound on the level at which the finite convergence is achieved~\cite{nie2014optimality}. Similar finite convergence results may be obtained in other situations, such as convexity~\cite{lasserre2009convexity,de2011lasserre}.
 \item In terms of asymptotic convergence rates (in dimension greater than one), they are quite slow, at best $O(1/s^2)$ in the simplest situations for the relaxation with polynomials of degree $s$~\cite{fang2021sum,laurent2022effective,slot2022sumCD}, {already improving on more generic results with rates in $O(1/s^c)$ for an unspecified value of $c$~\cite{schweighofer2004complexity,baldi2022effective,baldi2022ojasiewicz}}.
 \EIT
{Hierarchies for polynomial optimization come in two main types, using two different representations for non-negative polynomials under polynomial constraints. The ``Putinar representation'' adds as many polynomials as the number of constraints~\cite{putinar1993positive}, while the ``Schm\"udgen representation'' adds an exponential number~\cite{schmudgen2017moment}}.  In this paper, we focus on one of the simplest formulations of minimizing polynomials on $[-1,1]^d$ {with the Schm\"udgen representation}, which, as we show below through the use of Chebyshev polynomials, can be formulated as minimizing specific instances of trigonometric polynomials on $[0,1]^d$, which will be our primary focus, since for unconstrained optimization of trigonometric polynomials, most results simplify.
 
 We make the following contributions:
 \BIT
 \item We provide in \mysec{nocond} an $O(1/s^2)$ convergence result for the level of the hierarchy corresponding to trigonometric polynomials of degree $s$, \emph{without any assumptions}, that extends the work of~\cite{laurent2022effective} for polynomials on $[-1,1]^d$, with a similar proof technique (taken from~\cite{fang2021sum}), but with simpler arguments and explicit constants.
 \item When we add local optimality conditions similar to~\cite{nie2014optimality}, we prove in \mysec{withcond} an exponential convergence rate with explicit (but more complex) constants. The proof technique is taken from~\cite{rudi2020finding,woodworth2022non}, who showed convergence rates faster than any polynomial in $s$, but without explicit constants.
 \EIT
 
Our proof techniques deviate from previous work on polynomial hierarchies by focusing on the smoothness properties of the optimization problems rather than their algebraic properties. More precisely, this allows us (1) to use square roots and matrix square roots (which will typically lead to non-polynomial functions when taken on polynomials) together with their differentiability properties, and (2) to consider all infinitely differentiable functions with specific control of all derivatives, which trigonometric polynomials are only a sub-class of.

\section{Problem set-up}

\paragraph{Periodic functions and trigonometric polynomials}
 We consider $1$-periodic continuous functions $f$ on $\rb^d$, which we restrict to  $f: [0,1]^d \to \rb$, with summable Fourier series, that is, for which the ``F-norm'':
 \[
 \| f\|_{\rm F} = \sum_{\omega \in \zb^d} | \hat{f}(\omega)|
 \]
 is finite, where $\ds \hat{f}(\omega) = \int_{[0,1]^d} f(x) e^{-2i \pi \omega^\top x} dx $ is the Fourier series of $f$. We can then represent such functions as sums of complex exponentials $f(x) = \sum_{\omega \in \zb^d} \hat{f}(\omega) e^{2i\pi \omega^\top x}$, where the series is uniformly convergent. A key property of the F-norm is its relationship with the $L_\infty$-norm, that is, $\| f\|_\infty \leqslant \|f\|_{\rm F}$.
 
 We consider real-valued functions $f$, that is, such that $\hat{f}(-\omega) = \hat{f}^\ast(\omega)$ for all $\omega \in \zb^d$. This implies we can write $f(x)$ as real linear combinations of $\cos 2\pi \omega^\top x$ and $\sin 2\pi \omega^\top x$, and thus as a linear combination of monomials in $\cos 2\pi x_1,\dots,\cos 2\pi x_d,$ $ \sin 2\pi x_1,\dots,\sin 2\pi x_d$. This includes, but is not limited to,  trigonometric polynomials of degree $2r$, which corresponds to functions with vanishing Fourier series coefficients $\hat{f}(\omega)$ for $\|\omega\|_\infty > 2r$, that is,
 \[
 f(x) = \sum_{\| \omega \|_\infty  \leqslant 2r} \hat{f}(\omega) e^{2i\pi \omega^\top x}.
 \]
{We denote by $x_\ast$ any minimizer of $f$ on $[0,1]^d$ and by $f_\ast$ the minimal value (which does not depend on the chosen minimizer).}
\paragraph{Hierarchies of SOS optimization problems}
 We consider the maximization of $c$ such that $f-c$ is a sum of squares of trigonometric polynomials of degree $s$. We denote the optimal value by $c_\ast(f,s)$. The principle behind SOS hierarchies is that when $f$ is a trigonometric polynomial, this optimization problem can be solved as a finite-dimensional semi-definite programming (SDP) problem that we describe in \mysec{sdp} and thus be solved with a variety of algorithms~(see, e.g.,~\cite{helmberg1996interior}).
 
 If $f$ is a trigonometric polynomial of degree $2r$ with  $r \leqslant s$, then the value is finite, and we always have $c_\ast(f,s) \leqslant  \inf_{x \in [0,1]^d} f(x) = f_\ast$. Our main goal is to provide a bound:
 \BEQ
 \label{eq:EPS}
0 \leqslant  \inf_{x \in [0,1]^d} f(x) -c_\ast(f,s)  \leqslant \varepsilon(f,s),
 \EEQ
 depending on simple properties of $f$, and that tends to zero when $s$ tends to $+\infty$ with an explicit dependence in $s$.

 \paragraph{Beyond polynomials}
 When $f$ is not a trigonometric polynomial (of sufficiently low degree), then the SDP is not feasible (and the value thus equal to $-\infty$), but as shown in~\cite{woodworth2022non}, by using $c - \| f - c - g\|_{\rm F}$ as an objective function (with $g$ an SOS trigonometric polynomials of degree $2s$),  we always get feasible problems with values less than the minimal value of $f$. They can then be solved with appropriate sampling schemes (see~\cite{woodworth2022non} for details).

 \subsection{Semidefinite programming formulations}
 \label{sec:sdp}
In this section, we provide an explicit description of the semi-definite program for the SOS relaxation, as well as the associated spectral relaxation. For trigonometric polynomials, the optimization problems can be compactly written.

For an integer $s$, we consider the feature map $\varphi:[0,1]^d \to \cb^{(2s+1)^d}$, indexed by $\omega \in \{-s,\dots,s\}^d$ with values:
\BEQ
\label{eq:phi}
\varphi_\omega(x) = \frac{1}{(2s+1)^{d/2}} \exp( 2i \pi \omega^\top x).
\EEQ
It satisfies $\| \varphi(x) \|=1$ for all $x \in [0,1]^d$, where $\| \cdot\|$ denotes the standard Hermitian norm.

We can represent any trigonometric polynomial of degree $2s$ as a quadratic form in $\varphi(x)$, that is, we can write $f$ (non-uniquely) as $f(x) = \varphi(x)^\ast F \varphi(x)$, where $F$ is a Hermitian matrix of dimension $(2s+1)^d \times (2s+1)^d$. We denote by $\mathcal{V}_s$ the set of multivariate Hermitian Toeplitz matrices in dimension $(2s+1)^d \times (2s+1)^d$, that is, Hermitian matrices $\Sigma$ such that $\Sigma_{\omega \omega'}$ depends only $\omega-\omega' \in \zb^d$. It turns out that the span of all matrices $\varphi(x) \varphi(x)^\ast$ for $x \in [0,1]^d$ is exactly $\mathcal{V}_s$. We denote by $\mathcal{V}^\perp_s$ the orthogonal complement of $\mathcal{V}_s$ for the dot-product $(M,N) \mapsto \tr (M^\ast N)$.

\paragraph{Primal-dual formulations}

The SOS relaxation is obtained by solving
\[
 \max_{c \in \rb, \ A \succcurlyeq 0 } \ \ c   \ \  \mbox{ such that }\   \forall x \in [0,1]^d, \ f(x) = c + \varphi(x)^\ast A \varphi(x).
\]
It can be re-written using $\mathcal{V}_s$ as:
\BEA
\notag & &  \max_{c \in \rb, \ A \succcurlyeq 0 } \ \ c  \ \ \mbox{ such that }\ \ \forall x \in [0,1]^d, \ \tr \big[ \varphi(x) \varphi(x)^\ast ( F - c \idm - A) \big] = 0 \\
\notag & = & \max_{c \in \rb, \ A \succcurlyeq 0, \ Y \in   \mathcal{V}_s^\perp } \  c \  \ \mbox{ such that }\ \  F - c \idm - A + Y = 0 
\\
\label{eq:primalSOS} & = & \max_{Y \in  \mathcal{V}_s^\perp }\   \lambda_{\min}(F + Y),
\EEA
whose optimal value is $c_\ast(f,s)$.
Its dual can be written as, using standard semi-definite duality:
\BEA
\notag \max_{Y \in  \mathcal{V}_s^\perp }\   \lambda_{\min}(F + Y)
\notag  & = & \min_{ \Sigma \succcurlyeq 0 } \ 
\max_{Y \in  \mathcal{V}_s^\perp }\   \tr [ \Sigma(F + Y)] \ \mbox{ such that } \ \tr(\Sigma) = 1 \\
\label{eq:dualSOS} & = & \min_{ \Sigma \succcurlyeq 0 } \ 
   \tr ( \Sigma F )\ \mbox{ such that } \ \tr(\Sigma) = 1 , \ \Sigma \in \mathcal{V}_s,
\EEA
which corresponds to an outer approximation of the convex hull of all $\varphi(x)\varphi(x)^\ast$, $x \in [0,1]^d$, by the set of positive semi-definite matrices such that 
$\tr(\Sigma) = 1$ and $\Sigma \in \mathcal{V}_s$.

\paragraph{Spectral relaxation}
We can further relax the problem by equivalently setting $Y=0$ in \eq{primalSOS}, or removing the constraint $\Sigma \in \mathcal{V}_s$ in \eq{dualSOS}, and we simply obtain $\lambda_{\min}(F)$, which is the natural spectral relaxation of the minimization of $\varphi(x)^\ast F \varphi(x)$, by only considering that $\|\varphi(x)\|=1$. {This relaxation is appealing computationally as it can be solved in quadratic time in the dimension of $F$ as opposed to more than cubic for the SDP corresponding to the SOS problem}, but it leads in general to slow rates (see Appendix~\ref{app:spectral}).

\subsection{Relationship with polynomial hierarchies on $[-1,1]^d$}
\label{sec:poly}

In this section, we show how results on trigonometric polynomials on $[0,1]^d$ lead to results on regular polynomials on $[-1,1]^d$.

 Given a real polynomial $P$ on $\rb^d$ of degree $2r$, we define the function $f:[0,1]^d \to \rb$ as 
 \[
 f(y) = P(\cos 2\pi y_1,\dots,\cos 2\pi y_d),\]
  which is a trigonometric polynomial on $[0,1]^d$.
 
If the function $f$ is a sum of squares of trigonometric polynomials, it is the sum 
of terms of the form $\ds \big[ Q(\cos 2\pi y_1,\dots,\cos 2\pi y_d,\sin 2\pi y_1,\dots,\sin 2\pi y_d) \big]^2$, where $Q$ is a regular multivariate polynomial.

We can then use the unique decomposition of multivariate trigonometric polynomials as\footnote{This is a simple consequence of the definitions of Chebyshev polynomials of the first and second kinds~(see, e.g., \cite{dumitrescu2007positive}), that show that for $\omega\geqslant 1$, $\cos 2\pi \omega z$ is a polynomial in $\cos 2\pi z$, while $\sin 2\pi (\omega +1) z$ is the product of $\sin 2\pi z$ and a polynomial in $\cos 2\pi z$.} 
\BEAS
Q(\cos 2\pi y_1,\dots,\cos 2\pi y_d,\sin 2\pi y_1,\dots,\sin 2\pi y_d)  \\
& & \hspace*{-3cm} =  \sum_{J \subset \{1,\dots,d\} } Q_J(\cos 2\pi y_1,\dots,\cos 2\pi y_d) \prod_{j \in J} \sin 2\pi y_j,
\EEAS
where $Q_J$ is a multivariate polynomial.
Then, when taking the square, we get the following terms  for all $J,J' \subset \{1,\dots,d\}$:
\[
Q_J(\cos 2\pi y_1,\dots,\cos 2\pi y_d)Q_{J'}(\cos 2\pi y_1,\dots,\cos 2\pi y_d) \prod_{j \in J} \sin 2\pi y_j
\prod_{j' \in J'} \sin 2\pi y_{j'}.
\]
When $J=J'$, writing $x_1 = \cos 2\pi y_1, \dots, x_d = \cos 2\pi y_d$ for $x \in [-1,1]^d$, we get the term
\BEQ \label{eq:sm}
Q_J(x_1,\dots,x_d)^2 \prod_{j \in J} ( 1 - x_j^2),
\EEQ
while for $J \neq J'$, the sum of all terms coming from all squares must vanish because the original trigonometric polynomial $f$ has no sine terms.

Thus,  using Chebyshev polynomials, we get precisely the Schm\"udgen representation~\cite{schmudgen2017moment} of polynomials on $[-1,1]^d$, as the sum of terms of the form in \eq{sm} for all subsets $J \subset\{1,\dots,d\}$. Therefore, the existence of an SOS decomposition for $f$ leads to the existence of the corresponding Schm\"udgen representation for $P$ on $[-1,1]^d$. Thus our results also provide convergence rates for this hierarchy. We therefore actually extend results from~\cite{laurent2022effective}, which themselves provide a quantitative rate in $O(1/s^2)$, improving on the rates of the form $O(1/s^c)$, for an unspecified value of $c$, obtained in the more general set-up of all Schm\"udgen representations by~\cite{schweighofer2004complexity} (see~\cite{baldi2022effective} for a similar result for Putinar representations).

Note that our explicit results need to express a polynomial in the basis of Chebyshev polynomials, and then we consider the $\ell_1$-norm of the associated coefficients.

\paragraph{Transfer of local optimality conditions} While Theorem~\ref{theo:nocond} (\mysec{nocond}) will apply directly to regular polynomials with the construction above, Theorem~\ref{theo:2} (\mysec{withcond}) will require the function $f$ to have finitely many isolated second-order strict minimizers. {We show below that local second-order strict optimality conditions for the minimization of a regular polynomial on $[-1,1]^d$ translates to second-order strict optimality conditions for the corresponding problem on trigonometric polynomials.}

By symmetry, any $x \in (-1,1)^d$ is represented by $2^d$ potential $y$'s such that $x_i = \cos 2\pi y_i$, for $i \in\{1,\dots,d\}$, and if the minimum of $P$ on $[-1,1]^d$ is attained in $x_\ast$ in the interior $(-1,1)^d$, represented by $y_\ast \in [0,1]^d$ (any of the $2^d$ possible ones), we have $\frac{ \partial P}{\partial x_i}(x_\ast) = 0$ for all $i \in\{1,\dots,d\}$, and thus
$\frac{ \partial f}{\partial y_i}(y_\ast) \!=\! -2\pi \sin [2\pi (y_\ast)_i] \frac{ \partial P}{\partial x_i}(x_\ast) \!=\! 0 $, and 
\BEAS
\frac{ \partial^2 f}{\partial y_i \partial y_j}(y_\ast) & = &  
 - 1_{i=j} (2\pi)^2 \cos [2\pi (y_\ast)_i] \frac{ \partial P}{\partial x_i}(x_\ast)
 \\[-.25cm]
 & & \hspace*{4cm} +  (2\pi)^2  \sin [2\pi (y_\ast)_i]  \sin [2\pi (y_\ast)_j] \frac{ \partial^2 P}{\partial x_i \partial x_j}(x_\ast) \\[-.1cm]
 & = &   (2\pi)^2  \sin[ 2\pi (y_\ast)_i ] \sin [ 2\pi (y_\ast)_j] \frac{ \partial^2 P}{\partial x_i \partial x_j}(x_\ast).
 \EEAS
 Since $x_\ast \in (-1,1)^d$,   $\sin [2\pi (y_\ast)_i] \neq 0$ for all $i \in \{1,\dots,d\}$, and thus, if the Hessian of~$P$ at $x_\ast$ is positive definite, so is the one $f$ at $y_\ast$, and therefore we obtain $2^d$ strict second-order minimizers for the trigonometric polynomial if the original polynomial had such a minimizer in the interior of $[-1,1]^d$.
 
 If the minimizer $x_\ast$ is on the boundary, we obtain a similar result. Indeed, assume without loss of generality that $(x_\ast)_i = 1$ for $i \in \{1,\dots,r\}$ and $(x_\ast)_i \in (-1,1)$ for $i \in \{r+1,\dots,d\}$. We consider the following standard sufficient conditions for a strict local minimizer: $\frac{ \partial P}{\partial x_i}(x_\ast) < 0$ for $i \in \{1,\dots,r\}$, $\frac{ \partial P}{\partial x_i}(x_\ast) = 0$ for $i \in \{r+1,\dots,d\}$, and the square submatrix of the Hessian corresponding to indices in $ \{r+1,\dots,d\}$ is positive definite. Then, using the partial derivative computations above, we have $\frac{ \partial f}{\partial y_i}(y_\ast) = -2\pi \sin [2\pi (y_\ast)_i] \frac{ \partial P}{\partial x_i}(x_\ast)  = 0$ for all $i \in \{1,\dots,d\}$, since either  $\frac{ \partial P}{\partial x_i}(x_\ast)=0$ or  $\sin [2\pi (y_\ast)_i]=0$. Moreover, the Hessian of $f$ is block diagonal with one block composed of a diagonal matrix with elements $ - (2\pi)^2 \cos [2\pi (y_\ast)_i] \frac{ \partial P}{\partial x_i}(x_\ast)$ (which are strictly positive for $i \in \{1,\dots,r\}$), and another block with elements $ (2\pi)^2  \sin[ 2\pi (y_\ast)_i ] \sin [ 2\pi (y_\ast)_j] \frac{ \partial^2 P}{\partial x_i \partial x_j}(x_\ast)$, which is a positive definite block by assumption. Thus the Hessian is positive definite, and we obtain a second-order strict minimizer.

 \subsection{Review of existing results}
 In this section, we briefly review results about SOS hierarchies for the particular case of unconstrained optimization of trigonometric polynomials:
 \BIT
 \item If $d=1$, and $f$ is a trigonometric polynomial of degree $2r$, it is well-known that $\varepsilon(f,s) = 0$ as soon as $s \geqslant r$, as all non-negative trigonometric polynomials are sums-of-squares~\cite{fejer,riesz}.
 
 \item When $d=2$, then for any trigonometric polynomial $f$, the relaxation is tight with $s$ sufficiently large (but unknown a priori bound), that is $\varepsilon(f,s)$ is equal to zero for $s$ greater than some $s_0(f)$ (as a consequence of \cite[Corollary~3.4]{scheiderer2006sums}).
 
 \item When $d>1$, any \emph{strictly positive} trigonometric polynomial is a sum-of-squares~\cite{putinar1992complexification,megretski2003positivity}, but there exist non-negative polynomials which are not SOS~\cite{nafta}. Thus SOS hierarchies have to converge but cannot always be finitely convergent.
 \item When the set of zeroes of the non-negative function $f$ is finite and with invertible Hessians at these points, the hierarchy is finitely convergent, but with no a priori bound on the required degree~\cite{nie2014optimality}.
 \EIT

The goal of this paper is to provide upper-bounds of $\varepsilon(f,s)$ in \eq{EPS} for $d>1$, first without assumptions with a rate $O(1/s^2)$ (\mysec{nocond}), and then with stronger assumptions regarding the Hessian at optimum and explicit exponential rates (\mysec{withcond}).

\section{$O(1/s^2)$ convergence without assumptions for polynomials}
\label{sec:nocond}

We now show that the hierarchy of degree $s$ leads to a convergence rate in $O(1/s^2)$ with explicit simple constants and few assumptions. {Since no assumptions are made on polynomials except their degrees, this directly leads to an approximation result for moment matrices presented in \mysec{moment}.}\footnote{Sections~\ref{sec:nocond} and \ref{sec:withcond} are independent, and thus can read in any order.}

\begin{theorem}
\label{theo:nocond}
For any trigonometric polynomial $f$ of degree less than $2r$, we have, for any $s \geqslant 3r$, and for $\bar{f} = \hat{f}(0)$ the mean-value of $f$:
\[
\varepsilon(f,s) \leqslant  \| f - \bar{f}\|_{\rm F}\cdot   
\Big[  \Big( 1 - \frac{6r^2}{s^2} \Big)^{-d} - 1 \Big] \sim_{s \to +\infty} \| f - \bar{f} \|_{\rm F}\cdot    \frac{ 6 r^2 d }{s^2}.
\]
\end{theorem}
\begin{proof}
 We here follow the proof technique of~\cite{fang2021sum,laurent2022effective} based on integral operators by adapting it to trigonometric polynomials of degree $2r$, which are easier to deal with than spherical harmonics or regular polynomials through the use of Fourier series. We consider the following integral operator on $1$-periodic functions on  $ [0,1]^d$ to $ \mathbb{R}$, defined as 
 \BEQ
 \label{eq:Th}
 Th(x) = \int_{[0, 1]^d} |q(x-y)|^2 h(y) dy,
 \EEQ
 for a well-chosen $1$-periodic function $ q$ which is a trigonometric polynomial of degree~$ s$. 
 {The function $x \mapsto |q(x-y)|^2$ is an element of the finite-dimensional cone of SOS polynomials of degree $s$, thus,} by design, if $ h$ is a non-negative function, then $ Th$ is a sum of squares of polynomials of degree less than $ s$. We will find $ h$ such that $ Th = f - f_\ast+b$ for a constant $ b \geqslant 0$, for $f_\ast$ the minimal value of $f$, which will prove the result, since then $f = f_\ast - b + Th$, and $f_\ast-b$ is smaller than the value of the SOS relaxation $c_\ast(f,s)$, leading to $f_\ast - c_\ast(f,s) \leqslant b$.

In the Fourier domain, since convolutions lead to pointwise multiplication and vice-versa, we have for all $\omega \in \zb^d$, where $\hat{q} \ast \hat{q}(\omega)$ is a shorthand for $(\hat{q} \ast \hat{q})(\omega)$ : 
\[
\widehat{Th}(\omega) = \hat{q} \ast \hat{q}(\omega) \cdot \hat{h}(\omega),
\]
 and thus, the candidate $ h$ is defined by its Fourier series, which is equal to zero for $ \| \omega\|_\infty > 2r$, and to 
 \[ \frac{\hat{f}(\omega)+(b-f_\ast) 1_{\omega = 0}}{ \hat{q}  \ast \hat{q}(\omega)}\]
  otherwise. If we impose that $ \hat{q}  \ast \hat{q}(0)=1$, we then have 
\BEAS
   f -f_\ast + b  - h &  = &    \sum_{\omega  \in \zb^d} \hat{f}(\omega) \Big(  1  - \frac{1}{\hat{q} \ast \hat{q}(\omega)} \Big) \exp( 2i \pi \omega^\top \cdot) \\
& = &    \sum_{\omega  \neq 0 } \hat{f}(\omega) \Big(  1  - \frac{1}{\hat{q} \ast \hat{q}(\omega)} \Big) \exp( 2i \pi \omega^\top \cdot) .\EEAS
 We then get:
$ \| f -f_\ast + b  - h\|_\infty =  \big\| \sum_{\omega  \neq 0} \hat{f}(\omega) \big(  1  - \frac{1}{\hat{q} \ast \hat{q}(\omega)} \big) \exp( 2i \pi \omega^\top \cdot) \big\|_\infty.$

Using that $\|\cdot\|_{\infty} \leqslant \| \cdot\|_{\rm F}$, we get: 
\[ \| f -f_\ast + b \, - h\|_\infty \leqslant \!\sum_{\omega \neq 0}| \hat{f}(\omega)| \cdot  \! \max_{ \|\omega\|_\infty \leqslant 2r}  \Big|   \frac{1}{\hat{q} \ast \hat{q}(\omega)}  \, - 1\Big|
\leqslant \| f - \bar{f}\|_{\rm F}\cdot  \max_{ \|\omega\|_\infty \leqslant 2r}  \Big|   \frac{1}{\hat{q} \ast \hat{q}(\omega)}  \, - 1\Big| .
\]
The goal is now to find a good function $ q:[0,1]^d \to \rb$ with Fourier support within the  ball of radius $ s$, so that $ \hat{q}  \ast  \hat{q}(\omega)$ is close to $ 1$ for $ \| \omega\|_\infty \leqslant 2 r$, and simply check when $ \| f -f_\ast + b \, - h\|_\infty \leqslant b$.

A simple candidate is 
$\hat{q}(\omega)   = \frac{1}{(2s+1)^{d/2}} 1_{\|\omega\|_\infty \leqslant s}$, {based on a ``box kernel''};  we can then compute the convolution and obtain that  $\hat{q}\ast \hat{q}(\omega) = \prod_{i=1}^d \big( 1 - \frac{|\omega_i|}{2s+1} \big)_+ \geqslant \big( 1 - \frac{2r}{2s+1} \big)^d, $ leading to $b=  \| f - \bar{f}\|_{\rm F} \cdot \big[  \big( 1 - \frac{2r}{2s+1} \big)^{-d} - 1 \big].$ 
When $ s$ goes to infinity, we have the equivalent $ b  \sim  \| f - \bar{f}\|_{\rm F}\cdot    \frac{ rd }{s} = O(1/s)$, which thus converges to zero, but at a slow rate.

A better candidate leads to a rate in $ O(1/s^2)$ (like in~\cite{fang2021sum,laurent2022effective}), {is based on a a``triangular kernel''} as:
\[ \hat{q}(\omega) = a \prod_{i=1}^d \Big( 1 - \frac{|\omega_i|}{s} \Big)_+, \]
with $a$ a normalizing constant. 
A tedious computation including sums of powers of consecutive integers, {detailed in Appendix~\ref{app:tedious}}, leads to, for any $\|\omega\|_\infty\leqslant s$ (note that $\hat{q}\ast \hat{q}(\omega)$ is only equal to zero for $\|\omega\|_\infty>2s$),
\BEQ
\label{eq:tedious} \hat{q}\ast \hat{q}(\omega)
= a^2 \prod_{i=1}^d \Big[
\frac{2s}{3}+\frac{1}{3s}  - \frac{\omega_i^2}{s} + \frac{|\omega_i|}{2s^2}( \omega_i^2 - 1) \Big].
\EEQ
Thus we need $a^2 = \frac{1}{(\frac{2s}{3} + \frac{1}{3s})^d}$ to get $ \hat{q}\ast \hat{q}(0)=1$ and thus 
\[
\hat{q}\ast \hat{q}(\omega) 
\geqslant  \prod_{i=1}^d \Big( 1- \frac{1}{\frac{2s}{3} + \frac{1}{3s}} \frac{\omega_i^2}{s}\Big)_+
\geqslant \prod_{i=1}^d \Big( 1- \frac{3\omega_i^2}{2s^2}\Big)_+,
\]
 which is greater than $\big( 1- \frac{6r^2}{s^2}\big)_+^d$,
when in addition $\|\omega\|_\infty \leqslant 2r $.
This leads to, for $ s \geqslant 3 r \geqslant \sqrt{6} r $, 
\[ b  \leqslant  \| f - \bar{f}\|_{\rm F}\cdot   
\Big[  \Big( 1 - \frac{6r^2}{s^2} \Big)^{-d} - 1 \Big] \sim  \| f -  \bar{f}\|_{\rm F}\cdot    \frac{ 6 r^2 d }{s^2}.
\]
{Above, the asymptotic equivalent is taken with $s$ tending to infinity, with $r$ and $d$ being fixed.}
\end{proof}

We can make the following observations:
\BIT
\item The proposed bound follows a series of earlier bounds with similar behavior in $O(1/s^2)$ for the convergence rate of Lasserre's SOS hierarchies and uses the same proof technique based on integral operators~\cite{fang2021sum,laurent2022effective,slot2022sumCD,slot2022sum}. The most closely related is the one of~\cite{laurent2022effective}, which considers regular polynomials on $[-1,1]^d$ with  Schm\"udgen's representation, but with a different choice for the function $q$ in \eq{Th}. As shown in \mysec{poly}, our bound also applies to this case through a change of variable; it differs in the choice of normalization of coefficients (for us, $\ell_1$-norm of the expansion in Chebyshev polynomials).

\item Note that we could extend this result to other types of regularity beyond finite support and bounded F-norm, with the asymptotic bound $\| f - \bar{f}\|_{\rm F}\cdot    \frac{ 6 r^2 d }{s^2} + \sum_{\| \omega\|_\infty > 2r} | \hat{f}(\omega)|$, and by optimizing over $r \leqslant s$.

\item We believe the proof technique based on integral operators cannot lead to a better rate than $O(1/s^2)$, with the following informal argument. To obtain a faster rate in the simplest one-dimensional case, the function $r: [0,1] \to \rb$ defined as $r(x) = | q(x)|^2$, should be so that its Fourier series $\hat{r}(\omega)$ is of the form $f(\omega/s)$ for a function $f: \rb \to \rb$ such that $f''(0)=0$ and with support in $[-2,2]$. Thus, when $s$ gets large, $r(x) = \sum_{|\omega| \leqslant 2s } f(\omega/s) e^{2i\pi \omega x}$ should be proportional to the Fourier transform of~$f$. Thus the Fourier transform of $f$ should be non-negative with $f''(0) \propto \int_\rb x^2 \hat{f}(x)^2 dx = 0 $, which is impossible.

\item A natural open question is the optimality of the ``assumption-free'' bound in $O(1/s^2)$  (regardless of the proof technique). We show in the next section that adding extra assumptions leads to significantly better rates.

\item As shown in Appendix~\ref{app:spectral}, it turns out that a simple spectral relaxation of the problem already achieves a rate in $\|f - \bar{f}\|_{\rm F} \cdot \frac{rd}{s}$, which is worse than the $O(1/s^2)$ rate that we show in this section, but not representative of the empirical differences between the two methods. Our following result will show an explicit benefit of the SOS relaxation by obtaining exponential convergence rates (with extra assumptions on $f$).

\EIT

 \subsection{Approximation of moment matrices}
 \label{sec:moment}
 {We denote by $\mathcal{K}_s$ the closure of the convex hull of all Hermitian matrices $\varphi(x)\varphi(x)^\ast \in \cb^{(2s+1)^d \times (2s+1)^d}$ for $ x \in [0,1]^d$. It contains exactly all moment matrices; SOS relaxations can then be interpreted by relaxing it to the set $\widehat{\mathcal{K}}_s$ of ``pseudo-moment matrices''~$\Sigma$ such that $\Sigma \in \mathcal{V}_s$, $\Sigma \succcurlyeq 0$, and $\tr(\Sigma)=1$.
For $r\leqslant s$, we denote by ${\Pi}_{s}^{(r)}$ the linear operator on $ \mathcal{V}_s$ that sets of elements $H_{\omega\omega'}$ to zero as soon as $\| \omega\|_\infty > r $ or $\|\omega'\|_\infty > r$, and multiply all other elements by $(2s+1)^d / (2r+1)^d$ (making it essentially an element of $\mathcal{V}_r$). A classical duality argument leads to the following corollary of Theorem~\ref{theo:nocond}.  See proof in Appendix~\ref{app:hausd}.}

 \begin{corollary}
 \label{cor:hausd}
For any $s \geqslant 3r$, and any $\Sigma \in \widehat{\mathcal{K}}_s$, there exists $\Sigma' \in \mathcal{K}_s$ such that
\[
\big\| {\Pi}_{s}^{(r)} \big( \Sigma - \Sigma') \|_{\rm Frob}
\leqslant  \frac{\sqrt{2}}{(2r+1)^{d}} \Big[  \Big( 1 - \frac{6r^2}{s^2} \Big)^{-d} - 1 \Big],
\]
where $\| M \|_{\rm Frob}$ denotes the Frobenius norm of $M$.
 \end{corollary}
 {This corollary shows that matrices in $\mathcal{K}_r$ can be well approximated by projections of matrices in $ \widehat{\mathcal{K}}_s$. Note that the factor $(2r+1)^{-d}$ is an outcome of our choice of normalization to unit traces.}
 
 \section{Exponential convergence with local optimality conditions}
 \label{sec:withcond}
 We consider the simplest situation where the minimum of $f$ is attained at a unique point $x_\ast$ on the torus, and we assume that the Hessian $f''(x_\ast)$ is invertible. This implies that there exist ``conditioning'' constants $\alpha \in [0,1/2),\beta >0$, and $\lambda > 0 $ such that:
 \BEQ
\label{eq:conditioning} \| x - x_\ast\|_\infty \leqslant \alpha \Rightarrow f''(x) \succcurlyeq \lambda \idm  \ \mbox{ and }  \ 
 \| x - x_\ast\|_\infty \geqslant \frac{\alpha}{2} \Rightarrow f(x) - f(x_\ast) \geqslant \beta,
 \EEQ
 that is, (a) in the $\ell_\infty$-ball of radius $\alpha$ around $x_\ast$, the Hessian of $f$ has strictly positive eigenvalues greater than $\lambda$ (which we can take to be $\frac{1}{2} \lambda_{\min}(f''(x_\ast))$), and hence $f$ is strictly convex, and (b) away from a slightly smaller ball, $f-f(x_\ast)$ is strictly positive and greater than $\beta > 0$. See the illustration below in one dimension.
 
 \vspace*{.1cm}
 
 \begin{center}
 \includegraphics[width=8.5cm]{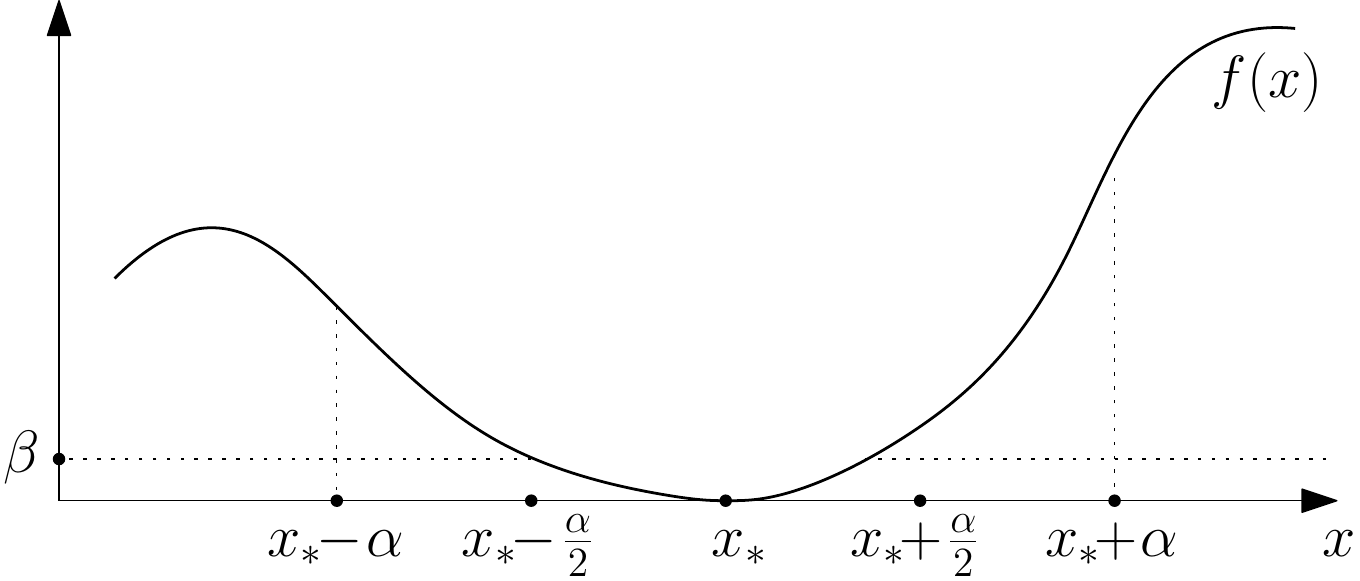}
  \end{center}
  
 \vspace*{.1cm}
 
 The proof technique is based on the one introduced in Lemma~1 and Theorem~2 of~\cite{rudi2020finding} (for the non-periodic case and without explicit constants) and can be extended directly to situations where the global minimum is attained at finitely many points with the same local Hessian condition (see also~\cite{ferey2022} for cases where minimizers are whole manifolds).
 
 Note that in that regime, the hierarchy is known to be finitely convergent~\cite{nie2014optimality}, but without bounds on the required degree $s$.
The following theorem gives an explicit bound on the convergence rate for any infinitely differentiable function with a specific growth condition for derivatives.
We denote by $\nabla^m f(x)$ the symmetric $m$-th order tensor of $m$-th order derivatives, {with element $\nabla^m f(x)_{j_1,\dots,j_m} = \frac{\partial^m f}{\partial x_{j_1} \cdots \partial x_{j_m}}(x)$, where $j_1, \dots,j_m \in \{1,\dots,d\}$.  Throughout the proofs, we will use the notation $\nabla^m f(x)[\delta,\dots,\delta] \in \rb$ to denote the contraction of the tensor along the $m$ copies of $\delta$, that is, 
$\nabla^m f(x)[\delta,\dots,\delta] = \sum_{j_1,\dots,j_m = 1}^d \nabla^m f(x)_{j_1,\dots,j_m} \delta_{j_1} \cdots \delta_{j_m}$.
We consider bounds on derivatives of the form 
\BEAS
\| \nabla^m f \|_\infty = \max_{x \in [0,1]^d} \sup_{\| \delta\|_1 \leqslant 1} |\nabla^m f(x)[\delta,\dots,\delta] |.
\EEAS
 Up to a constant that depends on $m$, this is equivalent to imposing a bound on all partial derivatives (see Appendix~\ref{app:A} for a precise relationship).
 This can also be seen as a bound on all directional derivatives, that is, of all $|g^{(m)}(0)|$, for $g(t) = f(x+t\delta)$.}

\begin{theorem}
\label{theo:2}
Assume that $f:[0,1]^d \to \rb$ is infinitely differentiable and such that $\| \nabla^m f \|_\infty \leqslant \| f - f_\ast\|_{\rm F} ( 4\pi r)^m $ for all $m \geqslant 0$.
Assume there exist $x_\ast \in [0,1]^d$, as well as, $\alpha \in [0,1/2),\beta >0$, and $\lambda > 0 $ such that \eq{conditioning} is satisfied. Then, we have: 
 \[
 \varepsilon(s,f) \leqslant  \triangle_1 \exp\Big( -  \Big(\frac{s}{\triangle_2 }\Big)^{1+\xi}  \Big),
 \]
 for any $\xi \in (0,1/2]$, with
 \BEQ
 \label{eq:triangle}
 \triangle_1 =  (  \beta +   \lambda d^3 )
\big( 32 B^3 d^6  \big)^{d+1}, \ \triangle_2 =  dB,
\EEQ
where $ \ds B =  
 \max\Big\{ \frac{275 }{\alpha\xi} , \frac{8 \pi r \| f - f_\ast\|_{\rm F}}{\beta}, \frac{6}{\lambda} \| f - f_\ast\|_{\rm F} (4\pi r)^3 \Big\}.
 $
\end{theorem}
Before describing the proof, we can make a few simple observations:
\BIT
\item Trigonometric polynomials of degree $2r$ satisfy the required growth condition, {because for $f: x \mapsto  e^{2i\pi\omega^\top x}$ where $\omega \in \zb^d$ such that $\|\omega\|_\infty \leqslant 2r$, we have:
$
\| \nabla^m f \|_\infty \leqslant (2\pi \cdot 2r)^m
$.}
\item The result extends a prior result~\cite{woodworth2022non}, that was showing convergence rates faster than any power of~$s$, but without explicit constants, which are needed to obtain the exponential rate.
{When the conditioning constant $\lambda$, $\alpha$, $\beta$ tend to zero, the constant $\Delta_2$
in \eq{triangle} tends to infinity, and the rate is not informative. In this situation, we could add a regularizer and optimize its strength to obtain a rate.}

\item  We could easily consider weaker growth conditions for the $m$-th order derivatives (with slower convergence rates), {such as $\| \nabla^m f \|_\infty = O( r^m m!)$.}

\item We could optimize over $\xi \in [0,1/2)$ to get a better dependence in $s$.

\item The result can be extended to functions with finitely many isolated second-order strict minimizers (following \cite[Theorem 2]{rudi2020finding}).
\EIT

 \subsection{Proof technique}
 \label{sec:mainproof}
 
 {The main technical result is to show that the non-negative function $f -f_\ast$ can be approximated by a trigonometric polynomial $g$ which is a sum of squares of polynomials of degree at most $s$, with an error bound measured in the norm~$\| \cdot \|_{\rm F}$ as $\| f - f_\ast - g\|_{\rm F} \leqslant \varepsilon'(s,f)$. Thanks to the following technical lemma whose proof is in Appendix~\ref{app:techlemma}, this leads to the desired result with $\varepsilon(s,f) = (2s+1)^{d} \varepsilon'(s,f) $.}
 
 \begin{lemma}
 \label{lemma:lemma42}
 {Assume $f$ is a trigonometric polynomial of degree less than $2s$, with minimal value $f_\ast$. If there exists
a trigonometric polynomial $g$ which is a sum of squares of polynomials of degree at most $s$ such that
$\| f - f_\ast - g\|_{\rm F} \leqslant \varepsilon'(s,f)$, then the optimal value $c_\ast(f,s)$ of \eq{primalSOS} and \eq{dualSOS} satisfies}
\[   0 \leqslant   f_\ast -  c_\ast(f,s)  \leqslant  (2s+1)^{d}\varepsilon'(s,f) .
\]
 \end{lemma}

{To obtain the desired approximant $g$, we follow the approach of~\cite{woodworth2022non} and builds an exact representation of $f-f_\ast$ as the sum of squared infinitely differentiable functions. We then truncate the Fourier series of these functions to obtain the approximation.}
 
 {To 
 provide the exact SOS decomposition, following~\cite{rudi2020finding}, we provide a decomposition around $x_\ast$, where the function $f-f_\ast$ has a zero, and away from $x_\ast$, where the function is strictly positive. This is then glued together with ``partitions of unity'' which we now present.
 }

 We consider two infinitely differentiable $1$-periodic functions $u,v: \rb^d \to [0,1]$ such that 
 \[
 \| x - x_\ast\|_\infty \leqslant \frac{\alpha}{2} \Rightarrow u(x)=1 
 \ \mbox{ and }  \  \| x - x_\ast\|_\infty \geqslant \alpha \Rightarrow u(x)=0,
 \]
 and for all $x \in \rb^d$,
 $
 u(x)^2 + v(x)^2 = 1$. See the illustration below in one dimension.
 
 \vspace*{.1cm}
 
 \begin{center}
 \includegraphics[width=8cm]{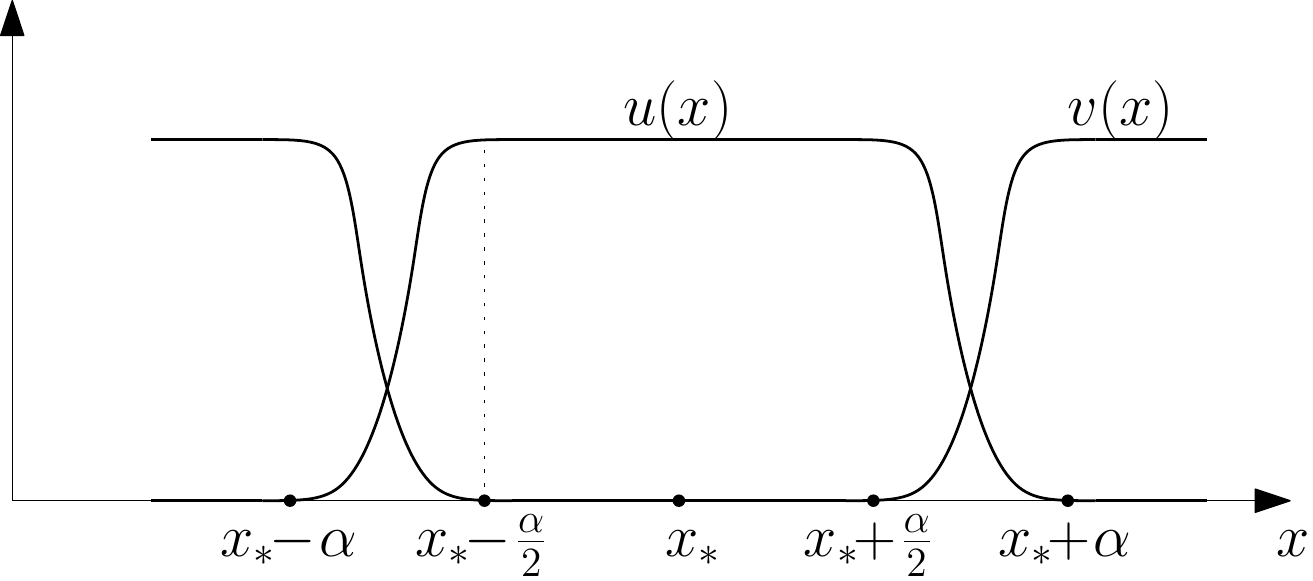}
  \end{center}
 
  \vspace*{.1cm}
  
 These are usually referred to as partitions of the unity and will be built in \mysec{boundpartition} using standard tools.
Following~\cite{rudi2020finding}, we can then decompose $f$ as, using Taylor's formula with integral remainder:
\BEAS
& & f(x) - f(x_\ast)  \\
 & = & v(x)^2 [ f(x) - f(x_\ast)  ] +  u(x)^2 [ f(x) - f(x_\ast)  ]
\\
& = & 
\Big[  v(x) \sqrt{f(x)-f(x_\ast)} \Big]^2
+ \Big[ u(x)^2
\int_0^1 (1-t) (x-x_\ast)^\top f''(x_\ast + t(x-x_\ast)) (x-x_\ast) dt
\Big]\\
& = &\Big[  v(x) \sqrt{f(x)-f(x_\ast)} \Big]^2
+ \Big[ u(x)^2(x-x_\ast)^\top R(x) (x-x_\ast) \Big]
\\[-.1cm]
& = &\Big[  v(x) \sqrt{f(x)-f(x_\ast)} \Big]^2
+ \Big[ u(x)^2 \sum_{i=1}^d (x-x_\ast)^\top R(x)^{1/2} u_i u_i^\top R(x)^{1/2} (x-x_\ast) \Big]
\\[-.25cm]
& = &\Big[  v(x) \sqrt{f(x)-f(x_\ast)} \Big]^2
+ \sum_{i=1}^d\Big[  u(x) (x-x_\ast)^\top R(x)^{1/2} z_i \Big]^2,
\EEAS
with $\ds R(x) = \int_0^1 (1-t)   f''(x_\ast + t(x-x_\ast))dt \succcurlyeq \frac{\lambda}{2}$ if $\| x - x_\ast\|_\infty \leqslant \alpha$, and $(z_1,\dots,z_d) \in \rb^{d \times d}$ any orthonormal basis of $\rb^d$.

We thus get an explicit SOS decomposition with $d+1$ functions as
\[
\forall x \in [0,1]^d, \ f(x) - f(x_\ast) = \sum_{i=1}^{d+1} g_i(x)^2,
\]
with 
\BEA
\label{eq:gi}
g_i(x) & = &  u(x) (x-x_\ast)^\top R(x)^{1/2} z_i \mbox{ for } i \in \{1,\dots,d\}, \\
\label{eq:gd1}
g_{d+1}(x) & = & v(x) \sqrt{f(x)-f(x_\ast)},
\EEA
which are infinitely differentiable functions (just taking the square root of $f-f_\ast$ without taking care of the region around the minimizer as we do above would not lead to a differentiable function).

We consider the truncations $\bar{g}_i$ obtained by keeping in $g_i$ only frequencies such that $\|\omega\|_\infty \leqslant s$, leading to, using lemmas from~\cite{woodworth2022non} about the F-norm (see also \cite[Section I.6]{katznelson2004introduction}):
\BEA
\notag \Big\| f - f_\ast - \sum_{i=1}^{d+1} \bar{g}_i^2 \Big\|_{\rm F} & = &  
\Big\| \sum_{i=1}^{d+1} g_i ^2 - \sum_{i=1}^{d+1} \bar{g}_i^2 \Big\|_{\rm F}
  \leqslant   \sum_{i=1}^{d+1} ( \| g_i\|_{\rm F} +  \| \bar{g}_i\|_{\rm F}  ) \| g_i - \bar{g}_i\|_{\rm F}
\\
\label{eq:boundfinal} & \leqslant & 2 \sum_{i=1}^{d+1}   \| g_i\|_{\rm F}  \cdot  \!\!\!\sum_{\| \omega \| > s} | \hat{g}_i(\omega)|
= 2 \sum_{i=1}^{d+1}   \| g_i\|_{\rm F}  \cdot  \| g_i\|_{{\rm F},s},
\EEA
where we denote $\| f\|_{{\rm F},s} =  \sum_{\| \omega \|_\infty > s} | \hat{ f } (\omega)|$. We thus need to find bounds on
$\| g_i\|_{\rm F} $ and $ \| g_i\|_{{\rm F},s}$, for $i \in \{1,\dots,d+1\}$, and then multiply the bound in \eq{boundfinal} above by the term $(2s+1)^{d}$ from Lemma~\ref{lemma:lemma42}.

Since these functions are $C^\infty$ (i.e., infinitely differentiable), the decay of their Fourier series is faster than any power, as already noted in \cite{woodworth2022non}. In the present paper, we provide explicit constants that allow us to obtain an exponential convergence rate.

{We will obtain bounds on Fourier series coefficients of the functions $g_i$ defined in \eq{gi} and \eq{gd1} by bounding their derivatives. Since they are defined as products, we need to bound the derivatives of each part: the partitions of unity $u$ and~$v$ (\mysec{boundpartition}), the scalar square root $(f-f_\ast)^{1/2}$
(\mysec{boundscalarroot}), and the matrix square root $R^{1/2}$ (\mysec{boundmatrixroot}). The bounds are then put together in \mysec{proofmain}.}

{The key in obtaining bounds on order $m$ derivatives is to track the dependence in $m$, with bounds of the form $c_1^m m! ^{1+c_2}$ for constants $c_1,c_2$.}

\subsection{Partitions of unity}
\label{sec:boundpartition}

Following~\cite[Section 3.1]{israel2015eigenvalue}, we
consider for $\eta \in (0,1]$, the function $a:\rb \to \rb$ defined as $a(x) = \exp(-(1-x^2)^{-1/\eta})$ on $[-1,1]$, and zero otherwise.  We then consider the function $b: \rb \to \rb$, defined as $ b(t) = \frac{\int_{-\infty}^t a(x)dx}{\int_{-\infty}^{+\infty} a(x)dx}$, which is non-decreasing, equal to zero for $t \leqslant 1$, and equal to $1$ if $t \geqslant 1$. These two functions are infinitely differentiable on $\rb$. See the illustrations below.

\vspace*{.1cm}

 \begin{center}
 \includegraphics[width=10.2cm]{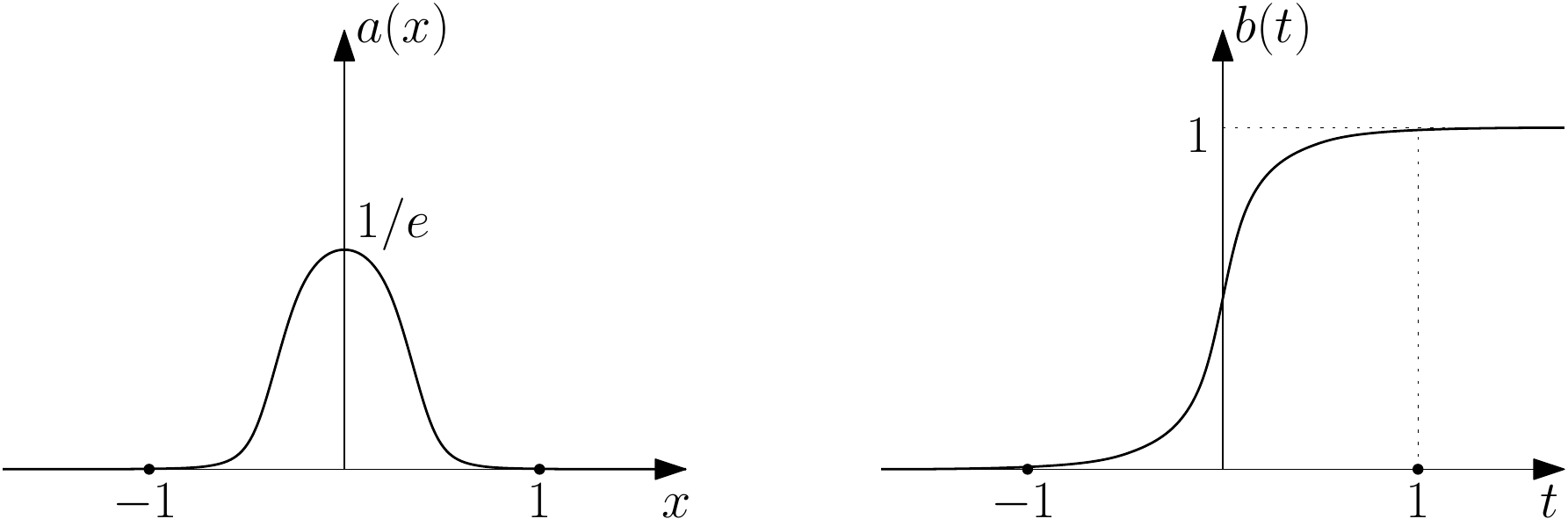}
  \end{center}

\vspace*{.1cm}

  We have, from \cite[Section 3.1]{israel2015eigenvalue},
   $|a^{(m)}(x)| \leqslant  \big( \frac{16}{\eta} \big)^m m^{(1+\eta)m}  $, for any $m\geqslant 0$, and any $x \in [-1,1]$. Moreover, we have
   \BEAS
   \int_{-\infty}^{+\infty} a(x)dx 
   & \geqslant & 2 \int_0^{\sqrt{\eta/2}}
    \exp(-(1-x^2)^{-1/\eta}) dx\\[-.1cm]
    & \geqslant & 
    \sqrt{2\eta}  \exp(-(1-\eta/2)^{-1/\eta})
    = \sqrt{2\eta} \exp\big( - \exp\big( -
    \frac{1}{\eta}\log\big(1-\frac{\eta}{2}\big) \big)\big).
   \EEAS
   Using $\log(1-x) \geqslant - (2 \log 2 )x$ for $x \in [0,1/2]$, we get the lower bound\footnote{Note that the bound from from~\cite{israel2015eigenvalue} is incorrectly independent of $\eta$.}
   \[
    \int_{-\infty}^{+\infty} a(x)dx 
   \geqslant
   \sqrt{2 \eta} \exp\big( - \exp\big( 
       \log 2 \big) \big)\big)  = \sqrt{2} e^{-2} \sqrt{\eta} \geqslant \sqrt{\eta}/8.
       \]
   We consider the function $w$ defined on $[-\frac{1}{2},\frac{1}{2}]$ as $\ds w(x) = b\big[ \frac{4}{\alpha} \big( |x| - \frac{3\alpha}{4}\big) \big]$, and extended by $1$-periodicity to~$\rb$. It is of the form plotted below.

\vspace*{.1cm}
   
    \begin{center}
 \includegraphics[width=8cm]{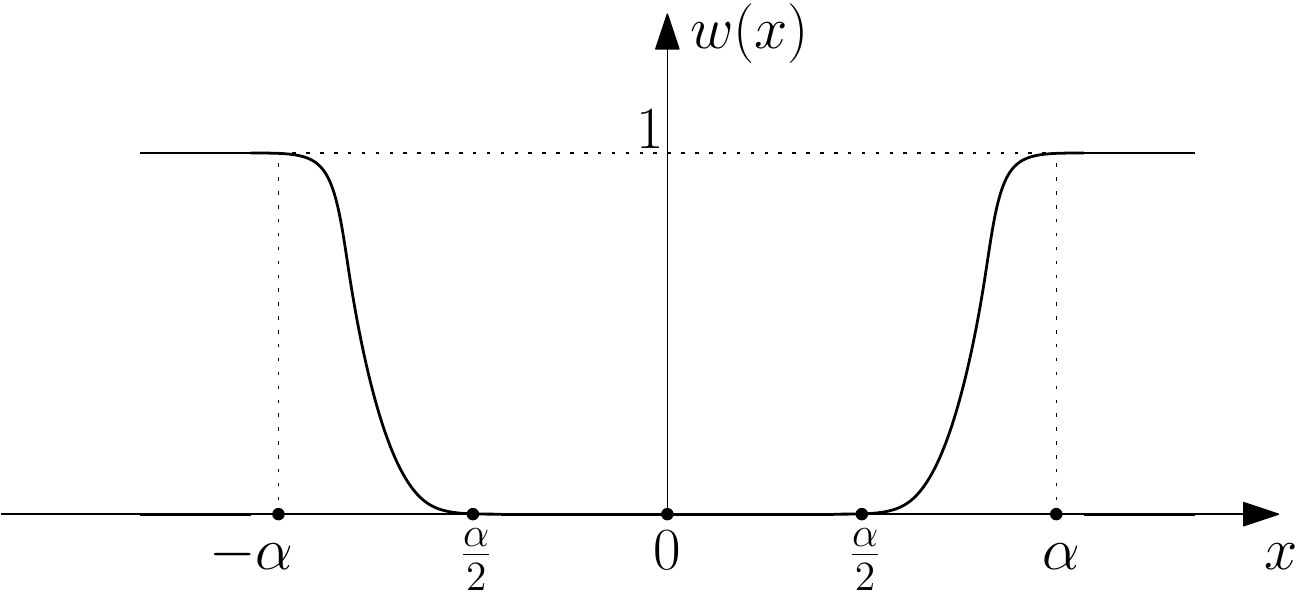}
  \end{center}

\vspace*{.1cm}

   Moreover  we have, {through the explicit expression of $w$ and the bounds on the derivatives of $a$ and $b$}:
   \[
   \forall x \in \big[-\frac{1}{2},\frac{1}{2}\big], \ | w^{(m+1)}(x)|  \leqslant  8 \sqrt{1/\eta} \big( \frac{64}{\alpha \eta} \big)^m m^{(1+\eta)m}  
   ,\]
   which leads to for $m>0$
   \[
   \forall x \in \big[-\frac{1}{2},\frac{1}{2}\big], \ | w^{(m)}(x)|  \leqslant  8 \sqrt{1/\eta} \frac{\alpha \eta}{64}  \big( \frac{64}{\alpha \eta} \big)^m m^{(1+\eta)m} \leqslant    c^m m^{(1+\eta)m}
   ,\]
   with $c = \frac{64}{\alpha \eta} $, an equality which is also valid for $m=0$ (where  $|w(x)|\leqslant 1$).

  We then consider the functions
  \BEA
  \label{eq:u}
  u(x) & = & \sin \Big[  \frac{\pi}{2} \prod_{i=1}^d \big( 1 - w(x_i-(x_\ast)_i)  \big) \Big] \\[-.1cm]
 \label{eq:v}  v(x) & = & \cos \Big[  \frac{\pi}{2} \prod_{i=1}^d \big( 1 - w(x_i-(x_\ast)_i)  \big) \Big].
  \EEA
  These functions satisfy exactly the constraints from \mysec{mainproof}, that is, $u(x)^2 + v(x)^2 = 1$ for all $x$, and, as soon as $\| x - x_\ast\|_\infty \leqslant \alpha/2$, $u(x) = \sin (\pi/2) = 1$, as well as, when $\| x - x_\ast\|_\infty \geqslant \alpha$, $u(x) = 0$. The next lemma provides bounds on their derivatives.
  
  \begin{lemma}
For the function $u$ defined in \eq{u} and \eq{v}, we have for all $m>0$:
\BEA
\label{eq:A}
\|  \nabla^m u\|_\infty & \leqslant &    \Big( \frac{275 }{\alpha \eta }  \Big)^{m}
m! \cdot m^{ \eta m},
\EEA
with the same bound for $v$ in \eq{v}.
\end{lemma}
\begin{proof}
We consider the function $g(t) = u(x+\delta t) = \sin \big[ \frac{\pi}{2} f(t) \big]$, with $f(t) =
\prod_{i=1}^d \big( 1 - w(x_i-(x_\ast)_i)  \big)$. We can expand the derivatives of the product function $f$ using the Leibniz formula to get for all $t$:
\BEAS
|f^{(m)}(t)|
& \leqslant  &\frac{\pi}{2} \bigg|\sum_{\alpha_1+\dots+\alpha_d = m}
{ m \choose \alpha_1,\dots,\alpha_d } \prod_{i=1}^d  c^{\alpha_i} \alpha_i^{(1+\eta)\alpha_i} \delta_i^{\alpha_i}  \bigg|
\\
& \leqslant & \frac{\pi}{2} \sum_{\alpha_1+\dots+\alpha_d = m}
{ m \choose \alpha_1,\dots,\alpha_d } \prod_{i=1}^d
\big[  c  m  ^{1+\eta}  | \delta_i| \big]^{\alpha_i} \\
&  = &  \frac{\pi}{2} \big[   cm^{1+\eta} \|\delta\|_1 \big]^m
\leqslant \frac{\pi}{2e}
  \big[   c e m^{\eta} \|\delta\|_1 \big]^m m! \ \ ,\mbox{ using } m^m \leqslant m! e^{m-1}.
\EEAS

We have, using Fa\`a di Bruno's formula (see, e.g.,~\cite{charalambides2018enumerative})  for the sine function, 
with the Bell polynomials $B_{m,k}$:
\BEAS
|g^{(m)}(t) |
& \leqslant &  \sum_{k=1}^m   B_{m,k} \Big(  \frac{\pi}{2e}
  \big[   ce1^{\eta} \|\delta\|_1 \big]^1 1!,\dots,  \frac{\pi}{2e }
  \big[ ce(m\!-\!k\!+\!1)^{\eta} \|\delta\|_1 \big]^{m\!-\!k\!+\!1} (m-k+1)!\Big)
\\
& \leqslant & \sum_{k=1}^m   B_{m,k} \Big(  \frac{\pi}{2e}
   \big[   c em^{\eta} \|\delta\|_1 \big]^1 1!,\dots,  \frac{\pi}{2e}
  \big[ c em ^{\eta} \|\delta\|_1 \big]^{m-k+1} (m-k+1)!\Big),
\EEAS
using the fact that Bell polynomials have non-negative coefficients (and are thus non-decreasing functions over the positive orthant).
Thus, using that 
\[ B_{m,k}(\alpha \beta z_1,\dots,\alpha \beta^{m-k+1} z_{m-k+1}) 
= \alpha^k \beta^m B_{m,k}(  z_1,\dots,  z_{m-k+1}),
\]
 we get, using an explicit formula for Bell polynomials,\footnote{See a summary of properties in \url{https://en.wikipedia.org/wiki/Bell_polynomials}.}
\BEAS
|g^{(m)}(t) |
& \leqslant & 
\sum_{k=1}^m
\big( \frac{\pi}{2e} \big)^k 
\big[ ce m^{\eta} \|\delta\|_1 \big]^m
  B_{m,k} \big( 1!,\dots,(m-k+1)! \big)
\\
& =  &  \big[ cem^{\eta} \|\delta\|_1 \big]^n
\sum_{k=1}^m 
\big( \frac{\pi}{2e} \big)^k 
 \frac{(m-1)!}{(k-1)!} { m \choose k},
\\
& \leqslant  &  \big[ cem^{\eta} \|\delta\|_1 \big]^n
\sum_{k=1}^m 
\big( \frac{\pi}{2e} \big)^k 
 m! { m \choose k} =  
 \big[ ce m^{\eta} \|\delta\|_1 \big]^m m! ( 1 + \pi / 2e)^m
 \\
 &  \leqslant &  \Big[ \frac{64e(1+\pi/(2e))}{\alpha \eta} m^{\eta} \|\delta\|_1 \Big]^m m! 
 \leqslant \Big[ \frac{275}{\alpha \eta} m^{\eta} \|\delta\|_1 \Big]^m m!  \ ,
\EEAS
which leads to $| \nabla^m u[ \delta,\dots,\delta]| \leqslant     \big[ \frac{275}{\alpha \eta} m^{\eta} \|\delta\|_1 \big]^m m!$, and thus to the desired result.
\end{proof}

   \subsection{Scalar square root}
\label{sec:boundscalarroot}
 {Since our SOS decomposition relies on the square root of the function $f-f_\ast$ for the function $g_{d+1}$ in \eq{gd1}, we need to bound square roots of functions which are strictly positive and bounded away from zero.} By applying  Lemma~\ref{lemma:squareroot} below to the function $g: t \mapsto f(x+t\delta) - f(x_\ast)$, for an arbitrary $\delta \in  \rb^d$ such that $\| \delta\|_1 \leqslant 1$ and $\| x - x_\ast\|_\infty \geqslant\frac{\alpha}{2}$, with $c = \beta$, $ C = \| f - f_\ast\|_{\rm F}$  and $ D = 4\pi r \|\delta\|_1$, we obtain that for $h: x \mapsto \sqrt{ f(x) - f_\ast}$:
\BEQ
 \label{eq:H}
  \| \nabla ^m h(x) \|_\infty = \max_{\| \delta\|_1 \leqslant 1 } |\nabla ^m h(x)[\delta,\dots, \delta]|
  \leqslant 3 \beta^{1/2} \Big( \frac{ 8 \pi r\| f - f_\ast\|_{\rm F}}{\beta} \Big)^k k! .
 \EEQ 
\begin{lemma}
\label{lemma:squareroot}
We consider a $C^\infty$ function $g$ defined on a neighborhood of zero (on the real line) such that $g(0) \geqslant c > 0 $ and such that for all $m \in \mathbb{N}$, 
$
|g^{(m)}(0)| \leqslant C \cdot D^m,
$
with $C \geqslant c$. For $b(t) = \sqrt{g(t)}$, we have:
$
|b^{(k)} (0)| \leqslant   3   c^{1/2}  \big(2 D \frac{C}{c} \big)^{  k }  k! \ $.
\end{lemma}
\begin{proof}
We will use Fa\`a di Bruno's formula (see, e.g.,~\cite{charalambides2018enumerative}), with the $k$-th derivative of $y \mapsto \sqrt{y}$ being $y^{\frac{1}{2}-k} (-1)^{k-1} \frac{1}{2k-1} \frac{ (2k)!}{(k)! 2^{2k}} = y^{\frac{1}{2}-k} b_k =  y^{\frac{1}{2}-k} k! C_{k-1} 2^{1-2k}$ for $k>0$, where $C_k = \frac{1}{k+1} { 2k \choose k}$ is the Catalan number. Using the classical bound $ C_n =  \frac{1}{n+1} \frac{(2n)!}{(n!)^2} \leqslant 2 \cdot 4^n$, we get
$
|b_k |  \leqslant
k!.
$
Fa\`a di Bruno's formula leads to, with the Bell polynomials $B_{k,i}$, and Stirling numbers of the second kind $s(k,i)$:
\BEAS
b^{(k)}(0)
& = & \sum_{i=1}^k g(0)^{\frac{1}{2} - i } b_{i}B_{k,i}(g'(0),\dots,g^{(k-i+1)}(0)).
\EEAS
This leads to 
\BEAS
| b^{(k)}(0)  | 
& \leqslant & \sum_{i=1}^k c^{\frac{1}{2} - i } i! B_{k,i}(C D,C D^2,\dots,C  D^{k-i+1})
= \sum_{i=1}^k D^k C^ i c^{\frac{1}{2} - i } i! B_{k,i}(1,1,\dots,1)
\\[-.25cm]
& = & D^k \sqrt{c} \sum_{i=1}^k  \Big(\frac{C}{c} \Big)^{  i }i!  |s(k,i)| \mbox{ using properties of Bell polynomials,}\\[-.2cm]
& \leqslant & D^k c^{1/2} \Big(\frac{C}{c} \Big)^{  k }  \sum_{i=0}^k   i!  |s(k,i)| \mbox{ which is the  ordered Bell number } A_k,
\\[-.05cm]
& \leqslant &  3   c^{1/2}  \Big(2 D \frac{C}{c} \Big)^{  k }  k! \ \, \EEAS
using the bound $A_k \frac{x^k}{k!} \leqslant \frac{1}{2-e^x}$, taken at $x = 1/2$.
\end{proof}

\subsection{Matrix square root}
\label{sec:boundmatrixroot}
{Since our SOS decomposition relies on matrix square roots for the functions $g_{1},\dots,g_d$ in \eq{gi}, we need the following lemma ($\| M\|_{\rm op}$ denotes the largest singular value of the matrix $M$), which can be seen as a matrix extension of Lemma~\ref{lemma:squareroot}.}
\begin{lemma}
\label{lemma:sqrtmat}
We consider a $C^\infty$ function $G: \rb \to \rb^{d \times d}$ with values in positive semidefinite matrices and defined on a neighborhood of zero (on the real line) such that $G(0) \succcurlyeq  c \idm$, with $c> 0$, and such that for all $m \in \mathbb{N}$, 
$
\| G^{(m)}(0) \|_{\rm op} \leqslant C \cdot D^m,
$
with $C \geqslant c$. For $h(x) = \tr [ M G(x)^{1/2} ]$, with $M$ a symmetric matrix such that $\| M\|_{\rm op}=1$, we have $
|h^{(k)} (0)| \leqslant   3   c^{1/2}  \big(2 D \frac{C}{c} \big)^{  k }  k!$.
\end{lemma}
\begin{proof}
We use results from~\cite{del2018taylor} and Lemma~\ref{lemma:fadi} below, with the operator norm on the set of symmetric matrices and the symmetric square root, where \cite[Theorem 1.1]{del2018taylor} exactly shows that we can take $\alpha(k) = \frac{k! C_{k-1} }{2^{2k-1}} c^{1/2-k}$, which is exactly the bound on $k$-th derivative of the square root which we used in Lemma~\ref{lemma:squareroot} above. Thus, the exact same derivations can be applied.
\end{proof}

\begin{lemma}
\label{lemma:fadi}
We consider functions $f: \rb^a \to \rb$ and $g: \rb \to \rb^a$, and $\varphi = f \circ g: \rb \to \rb$ that are infinitely differentiable. For a certain norm $\| \cdot\|$ on $\rb^a$, we assume that
\[
| \nabla^k f(g(0))[\delta_1,\dots,\delta_k] | \leqslant \alpha(k) \|\delta_1\| \cdots \| \delta_k\|,
\]
for some $\alpha(k)>0$.
Then for any $n \geqslant 1$, for the Bell polynomials $B_{n,k}$, we have:
\[
|\varphi^{(n)}(0)| 
\leqslant \sum_{k=1}^n \alpha(k) B_{n,k}\big( \| g^{(1)}(0)\|, \dots, \| g^{(n-k+1)}(0)\| \big).
\]
\end{lemma}
\begin{proof}
We follow the proof of Fa\`a di Bruno's formula that considers a Taylor expansion of $g$ around zero as, for any $m >0$:
$ 
g(t) - g(0) = \sum_{k = 1}^m \frac{t^k}{k!} g^{(k)}(0),
$
and of $f$ around $g(0)$, as
$ 
f(g(0)+\delta)
- f(g(0)) = \sum_{k=1}^m \frac{1}{k!} \nabla^k f(g(0)) [ \delta, \dots, \delta].
$
Thus $f(g(t))$ can be expanded as a polynomial in $t$, with coefficients composed of factors of the form 
$ c \nabla^{k} f(g(0))[g^{\alpha_1}(0),\dots,g^{\alpha_k}(0) ]$, with a \emph{non-negative} coefficient $c$. Each of them can then be bounded by the term
 $c \alpha(k) \|g^{\alpha_1}(0)\| \cdots \| g^{\alpha_k}(0)\|$, which is then equivalent to the formula obtained by applying the univariate Fa\`a di Bruno's formula, with a function with derivatives $\alpha(k)$, and the other one with derivatives $\|g^{k}(0)\|$. We then use the usual formulation with Bell polynomials.
\end{proof}

We can now apply it to bound derivatives of $g: x \mapsto (x-x_\ast)^\top R(x)^{1/2} z$ for $z \in \rb^d$. We consider $\varphi(t) = g(x+t\delta)$. We have, using the Leibniz formula:
\BEQ
\label{eq:Q}
\varphi^{(m)}(0) = (x-x_\ast)^\top \frac{\partial^m }{\partial t^m}
R(x+t\delta)^{1/2} z
+ m \delta^\top \frac{\partial^{m-1} }{\partial t^{m-1}}
R(x+t\delta)^{1/2} z.
\EEQ
We have, from expressions in \mysec{mainproof}, with $h(t) = R(x+t\delta) $,
\[
 h(t) = R(x+t\delta) = \int_0^1 (1-s)   f''(x_\ast + s(x+t\delta-x_\ast))ds,
\]
with derivatives which can be computed as, for any $v \in \rb^d$:
\[
 v^\top h^{(m)}(0) v = \int_0^1 (1-s)   \nabla^{m+2}f(x_\ast + s(x+t\delta-x_\ast))[ \delta, \dots,\delta,v,v] s^m ds.
 \]
Using assumptions from Theorem~\ref{theo:2}, in operator norm, $h^{(m)}(0)$ is less than the supremum over $\|v\|_2=1$ of (using $\| v\|_1^2 \leqslant d \|v\|_2^2$ and  integration):
\[
 \| f - f_\ast\|_{\rm F} (4\pi r )^{m+2}\|\delta\|_1^m \| v\|_1^2
\int_0^1 (1-s) s^m ds
\leqslant \| f - f_\ast\|_{\rm F} (4\pi r )^{m+2}\|\delta\|_1^m \frac{d}{m^2}.
\]
{Moreover, we have 
$h(t) \succcurlyeq \int_0^1 (1-s) \lambda \idm ds = \frac{\lambda}{2} \idm$.}
This leads to constants $c = \frac{\lambda}{2}$, $C =  \| f - f_\ast\|_{\rm F} (4\pi r)^2  d $ and
$D = 4 \pi r \| \delta\|_1$ for the function $h$, and thus, to the function $g: x \mapsto (x-x_\ast)^\top R(x)^{1/2} z$, with all derivatives of order $m$ less than (using Lemma~\ref{lemma:sqrtmat} and \eq{Q}):
\BEA
\notag & & \sqrt{d} \cdot 3 \sqrt{\lambda/2} \Big( \frac{4}{\lambda} \| f \!-\! f_\ast\|_{\rm F} (4\pi r)^3  d \Big)^m \! m!
+ m \cdot 3 \sqrt{\lambda/2} \Big( \frac{4}{\lambda} \| f\!- \! f_\ast\|_{\rm F} (4\pi r)^3  d \Big)^{m-1}\! (m\!-\!1)!
 ,
 \EEA
 which is less than
 \BEQ
 \label{eq:I}
 d \sqrt{\lambda}
\Big( \frac{6}{\lambda} \| f - f_\ast\|_{\rm F} (4\pi r)^3  d \Big)^m m! \ \ .
\EEQ

\subsection{Precise bound}
\label{sec:proofmain}
{We start with bounds on all derivatives of functions $g_i$, $i=1,\dots,d+1$, defined in \eq{gi} and \eq{gd1}, and then translate them into bounds on their Fourier series coefficients and thus $\|g_i\|_{\rm F}$ and $\|g_i\|_{\rm F,s}$.}

 \label{sec:details}
\label{sec:detail}
{To get our bound, we first realize that all of these functions are products of two functions, and thus we can use 
Lemma~\ref{lemma:prod} below, proved in Appendix~\ref{app:proofder}, that bounds derivatives of products.}

\begin{lemma}[Derivatives of products]
\label{lemma:prod}
\label{lemma:product}
Assume that $h_1,h_2: [0,1]^d \to \rb$ is $C^\infty$ and such  $\forall m \geqslant 0$, 
$\| \nabla^m h_1\|_\infty   \leqslant  C_1 \cdot B_1^m     \cdot m! \cdot \kappa_1(m)  ,
$ and $\| \nabla^m h_2\|_\infty  \leqslant  C_2 \cdot B_2^m \cdot  m! \cdot \kappa_2(m)  ,
$
Then 
$
\| \nabla^m (h_1 h_2) \|_\infty  \leqslant  C_1 C_2 \kappa_1(m) \kappa_2(m)     (m+1)! \max \{ B_1,B_2\}^m.
$
\end{lemma}
With the estimates in \eq{A} and \eq{I}, we get:
\BEQ
\label{eq:AAA}
\forall i \in \{1,\dots,d\}, \ 
\| \nabla^m g_{i} \|_\infty   \leqslant 
d \sqrt{\lambda}
\max\Big\{ \frac{275 }{\alpha \eta} ,  \frac{6}{\lambda} \| f - f_\ast\|_{\rm F} (4\pi r)^3 \Big\}^m  m! \cdot m^{ \eta m}.
\EEQ
For $g_{d+1}$, we need to consider two cases: one where $v$ is uniformly zero, and thus $g_{d+1}$ is zero as well, and one where $v$ is strictly positive, where $f-f_\ast$ is lower-bounded by~$\beta$, and we can apply bounds on derivatives of products. We thus get explicit bounds on all derivatives, from \eq{A} and \eq{H}:
\BEQ
\label{eq:BBB}
\|\nabla^m g_{d+1}\|_\infty  \leqslant     3 \beta^{1/2} \max\Big\{ \frac{275 }{\alpha\eta} , \frac{8 \pi r \| f - f_\ast\|_{\rm F}}{\beta} \Big\}^m  
m! \cdot m^{ \eta m}  .
\EEQ

We can now use Lemma~\ref{lemma:derdec} below (see proof in Appendix~\ref{app:proofderdec}) that relates the growth of derivatives to the (truncated) F-norm.
\begin{lemma}[From derivatives to Fourier decay]
\label{lemma:derdec}
Assume that $g: [0,1]^d \to \rb$ is~$C^\infty$ and such that for all $m \geqslant 0$, 
$
\| \nabla^m g \|_\infty  \leqslant  C \cdot B^m \cdot    m! \cdot \kappa(m)  ,
$
with $\kappa$ non-decreasing. 
Then, for $k \geqslant d+1$,
\BEAS
\| g\|_{\rm F} & \leqslant &  C \Big( 2 + \frac{d Bk}{ 2 \pi} \Big)^k \kappa(k)  \cdot  2 (2e)^{d-2}\\
\|  g \|_{{\rm F},s} & \leqslant & 
C \Big( 2 + \frac{d Bk}{ 2 \pi} \Big)^k \kappa(k)
 2 (2e)^{d-2} (s+1)^{d-k}. 
\EEAS
\end{lemma}

With $\ds B = \max\Big\{ \frac{275 }{\alpha \eta} , \frac{8 \pi r \| f - f_\ast\|_{\rm F}}{\beta}, \frac{6}{\lambda} \| f - f_\ast\|_{\rm F} (4\pi r)^3 \Big\} \geqslant 275$, we get from Lemma~\ref{lemma:derdec} above, \eq{AAA}, and \eq{BBB}, for all $k \geqslant d+1$:
\BEAS
\| g_{d+1}\|_{\rm F}
& \leqslant &  3 \beta^{1/2} \Big( 2 + \frac{B d  (d+1)}{ 2 \pi} \Big)^{d+1} (d+1)^{\eta(d+1)}  \cdot  2 (2e)^{d-2}
\\
\| g_{d+1}\|_{\rm F, s }
& \leqslant &  3 \beta^{1/2} \Big( 2 + \frac{d Bk}{ 2 \pi} \Big)^{k} k^{\eta k}  \cdot   2 (2e)^{d-2} s^{d-k}   
\\
\| g_{i}\|_{\rm F}
& \leqslant &  d \sqrt{\lambda} \Big( 2 + \frac{Bd (d+1)}{ 2 \pi} \Big)^{d+1} (d+1)^{\eta(d+1)}  \cdot  2 (2e)^{d-2}
\\
\| g_{i}\|_{\rm F, s }
& \leqslant &  d \sqrt{\lambda} \Big( 2 + \frac{d Bk}{ 2 \pi} \Big)^{k} k^{\eta k}  \cdot   2 (2e)^{d-2} s^{d-k}  ,
\EEAS
and thus a bound from \eq{boundfinal}:
\BEAS 
c & \leqslant &8  ( 9 \beta +   \lambda d^3 )
\Big( 2 + \frac{Bd (d+1)}{ 2 \pi} \Big)^{d+1}
 \Big( 2 + \frac{d Bk}{ 2 \pi} \Big)^{k}
  (2e)^{2d-4} s^{d-k} k^{\eta k}
    \\
& \leqslant &     (  \beta +   \lambda d^3 )
\big( 6 Bd^2  \big)^{d+1}
 \Big( \frac{d Bk}{ 6} \Big)^{k}
   s^{d-k} k^{\eta k}
    .
\EEAS
The main term is of the form
$
 \big( \frac{ k^{1+\eta} d B}{6 s} \big)^k.
$
 We then select $k = \big( \frac{ 6 s}{ e d B} \big)^{1/(1+\eta)}$, leading to the term
\[
 \exp\big( -  \big( \frac{ 6 s}{ e d B} \big)^{1/(1+\eta)} \big) \leqslant  \exp\big( -  \big( \frac{ 2 s}{ d B} \big)^{1/(1+\eta)} \big) .
\]
Overall,using the identity $ e^{- z} \leqslant \big( \frac{c}{ez} \big)^c$, applied to $c = \frac{3d}{2} $ and
$z = \frac{2}{5} \big(\frac{s}{dB}\big)^{1/(1+\eta)}$,  and  multiplying the bound in \eq{boundfinal} above by the term $(2s+1)^{d}$ from Lemma~\ref{lemma:lemma42}, and using $\eta \in (0,1]$, we get:
\BEAS 
\varepsilon(f,s) & \leqslant &  (2s+1)^{d}  (  \beta +   \lambda d^3 )
\big( 6 Bd^2  \big)^{d+1}
  s^d 
   \exp\big( -  \big( \frac{ 2 s}{ d B} \big)^{1/(1+\eta)} \big) \\
    & \leqslant &     (  \beta +   \lambda d^3 )
\big( 9 Bd^2  \big)^{d+1}
  s^{3d/2} 
   \exp\big( -  \big( \frac{ 2 s}{ d B} \big)^{1/(1+\eta)} \big) \\
  & \leqslant &     (  \beta +   \lambda d^3 )
\big( 9 Bd^2  \big)^{d+1}
  s^{2d} 
   \exp\big( -  \big( \frac{   s}{ d B} \big)^{1/(1+\eta)} \big) 
    \exp\big( -  \frac{2}{5} \big( \frac{   s}{ d B} \big)^{1/(1+\eta)} \big) \\
  & \leqslant &   (  \beta +   \lambda d^3 )
\big( 9 Bd^2  \big)^{d+1}  s^{2d} 
\Big(  \frac{15}{4e}  (dB/s)^{1/(1+\eta)} \Big)^{2d}
   \exp\big( -  \big( \frac{  s}{ d B} \big)^{1/(1+\eta)} \big) 
   \\
     & \leqslant &   (  \beta +   \lambda d^3 )
\big( 9 Bd^2  \big)^{d+1}
\Big(   \frac{15}{4e} d^2 B \Big)^{2d}
   \exp\big( -  \big( \frac{  s}{ d B} \big)^{1/(1+\eta)} \big)\\
        & \leqslant &   (  \beta +   \lambda d^3 )
\big( 32 B^3 d^6  \big)^{d+1}
   \exp\big( -  \big( \frac{  s}{ d B} \big)^{1/(1+\eta)} \big).
\EEAS
We then consider $\xi = 1 - \frac{1}{1+\eta} \in (0,1/2]$ to obtain the constants in \eq{triangle}.

\section{Discussion}

Our convergence results could be extended in several ways:
\BIT
\item While explicit polynomial convergence rates already exist for the Boolean hypercube~\cite{slot2022sum}, it would be interesting to obtain improved rates with some form of local condition.

\item Our proof technique relies on Fourier series {and the characterization of various orders of differentiability using the corresponding orthonormal basis.} It could thus be extended to all cases where such tools can be used, such as on the Euclidean hypersphere~\cite{fang2021sum} and beyond~\cite{rudin2017fourier}.

\item Almost all the techniques that we used to derive explicit constants can be extended easily to the more general kernel case \cite{rudi2020finding} (noting that the function~$q$ that we used is a specific instance of a translation-invariant periodic kernel), as well as the case where minimizers are manifolds~\cite{ferey2022}.

\item It would be interesting to extend our second result to provide an explicit bound on the degree for finite convergence.

\item We only focused on the unconstrained global optimization problem, but adding constraints and extending to more general problems (e.g., optimal control and optimal transport) is natural.

\EIT

\subsection*{Acknowledgements}
We thank Monique Laurent, Jean-Bernard Lasserre, and Milan Korda for helpful discussions about this work. Comments of the anonymous reviewers were greatly appreciated. We acknowledge support
from the French government under the management of the Agence Nationale de la {Recher-che} as part of the ``Investissements d’avenir'' program, reference ANR-19-P3IA0001 (PRAIRIE 3IA Institute). This work was also supported by the European Research Council (grants SEQUOIA 724063 and REAL 947908).

\appendix

\section{Computation of convolutions}
\label{app:tedious}

{Given the function $b: \mathbb{Z} \to \rb$ defined as $b(\omega)  = ( s - | \omega |)_+$, we  need to compute the convolution $b\! \ast\! b(\omega)$ for $|\omega| \leqslant s$. Since $b$ is even, so is $b \!\ast\! b$ and we can thus consider $\omega \in [0,s]$. We want to show that $b \!\ast\! b(\omega) = \frac{s(2s^2+1)}{3}  - \frac{\omega}{2} - s \omega^2 + \frac{\omega^3}{2}$.}

{We can split the sum 
$b \!\ast\! b(\omega) = \sum_{i\in \mathbb{Z}} b(i) b(\omega - i) $ as follows,
\BEAS
&   & \sum_{i = \omega - s}^0 ( s + i) ( s - \omega + i)
+
\sum_{i= 1}^\omega ( s - i) ( s - \omega + i)
+
\sum_{i = \omega +1 }^s ( s  - i) ( s - i + \omega)
\\
& = & \sum_{i = \omega - s}^0 \!\! \big[ i^2 + (2s-\omega)i + s(s-\omega) \big]
+
\sum_{i = 1}^\omega \big[ - i^2 +\omega i + s(s-\omega) \big]
\\[-.35cm]
& &\hspace*{7cm}
+
\sum_{i = \omega +1 }^s \!\! \big[ i^2 - (2s+\omega)i + s(s+\omega) \big].
\EEAS
Then, 
using $\sum_{i=1}^t \! i = \frac{t(t+1)}{2}= \frac{1}{2}(t^2+t)$ and $\sum_{i=1}^t \! i^2 = \frac{t(t+1)(2t+1)}{6} = \frac{1}{6}(2 t^3+3t^2 + t)$, we get:
\BEAS
&  & \frac{1}{6}(s-\omega)(s-\omega+1) ( 2s-2\omega+1)
- \frac{1}{2} (s-\omega)(s-\omega+1) ( 2s-\omega)\\
& & + s(s-\omega)(s-\omega+1)   -  2 \frac{1}{6} \omega(\omega+1)( 2\omega+1)
+ \frac{1}{2} \omega(\omega+1)(2s+2\omega) + s(s-\omega) \omega \\
& & + \frac{1}{6} s(s+1)( 2s+1)
- \frac{1}{2} s(s +1) ( 2s+\omega) + s(s+\omega)(s-\omega),
\EEAS
leading to
\BEAS
& & \frac{1}{6} (s-\omega)(s-\omega+1) (  \omega+1 - 4s )
+ s(s-\omega) ( 2s +\omega + 1  ) \\
& & + \frac{1}{6}(2 s^3+3s^2 +s ) - \frac{1}{3}(2 \omega^3+3\omega^2 +\omega )
+( \omega^2+\omega)(s+\omega) - \frac{1}{2}(s^2+s)(2s+\omega) \\
& = &\big[   \frac{1}{6}   -\frac{2}{3} + 1\big] \omega^3 +\big[  \frac{1}{6}(  {1-4s-s-1-s}) -s - 1+ s+1\big]\omega^2  \\
& & +
\big[ \frac{1}{6}(  (s+1)(4s\!-\!1)-s(1\!-\!4s)+s(s\!+\!1)) + s^2 -s(2s\!+\!1) - \frac{1}{3}+s  - \frac{1}{2}(s^2\!+\! s) \big] \omega  \\
& & + \big[ 
s(s+1)\frac{1}{6}(1-4s)+s^2(2s+1)+\frac{1}{6}(2 s^3+3s^2 +s ) -s(s^2+s)
\big]
\\
& = &\frac{1}{2}\omega^3 - s \omega^2 - \frac{1}{2}\omega +\frac{s}{3} + \frac{2}{3}s^3. 
\EEAS
To get \eq{tedious}, we then use $\hat{q} \ast \hat{q}(\omega) = a^2 \prod_{i=1}^d \frac{1}{s^2} b \!\ast\! b(|\omega_i|)$.}

\section{Performance of the spectral relaxation}
\label{app:spectral}
Given a trigonometric polynomial $f$ of degree $2r$, with $r \leqslant s$, we  can represent it as a quadratic form in $\varphi(x)$ defined in \eq{phi} as:
\[
f(x) = \varphi(x)^\top F \varphi(x) \mbox{ with } 
F_{\omega \omega'} = \hat{f}(\omega-\omega') \prod_{i=1}^d \Big( 1 - \frac{|\omega_i - \omega_i'|}{2s+1} \Big)^{-1},
\]
which is the unique Toeplitz representation $F$ for $f$.
We denote by $g: [0,1]^d \to \rb$ the function with Fourier series $\hat{g}(\omega) = \hat{f}(\omega) \prod_{i=1}^d \big( 1 - \frac{|\omega_i|}{2s+1} \big)^{-1}$.

For any $z \in \cb^{(2s+1)^d}$ of unit norm, we have:
\BEAS
z^\ast F z
& = & \sum_{\| \omega \|_\infty, \| \omega' \|_\infty \leqslant s } z_\omega z_{\omega'}^\ast 
\int_{[0,1]^d} g(x) \exp( -2i\pi (\omega-\omega')^\top x )dx \\[-.1cm]
& = & \int_{[0,1]^d} g(x) \bigg|
 \sum_{\| \omega \|_\infty  \leqslant s }
 z_\omega   \exp( -2i\pi  \omega ^\top x )
\bigg|^2 dx \\[-.1cm]
& \geqslant  &   \inf_{x' \in [0,1]^d} g(x') \cdot \int_{[0,1]^d}  \bigg|
 \sum_{\| \omega \|_\infty  \leqslant s }
 z_\omega   \exp( -2i\pi  \omega ^\top x )
\bigg|^2 dx =  \inf_{x' \in [0,1]^d} g(x').
\EEAS
Thus $\lambda_{\min}(F) \geqslant  \inf_{x \in [0,1]^d} g(x) $. We have moreover:
\BEAS
\| f - g\|_\infty & \leqslant &   \sum_{\omega \in \zb^d}
 |\hat{f}(\omega)| \cdot \Big| \prod_{i=1}^d \big( 1 - \frac{|\omega_i|}{2s+1} \big)^{-1} - 1 \Big|  \\[-.1cm]
 & \leqslant & \| f - \bar{f}\|_{\rm F}\Big[ \Big(  1 - \frac{2r}{2s+1}\Big)^{-d} - 1\Big] \sim_{s \to +\infty}
 \| f - \bar{f} \|_{\rm F} \cdot \frac{ rd}{s},
\EEAS
which leads to 
 \[
 0 \geqslant \lambda_{\min}(F) -f_\ast \geqslant  - 
  \| f - \bar{f}\|_{\rm F}\Big[ \Big(  1 - \frac{2r}{2s+1}\Big)^{-d} - 1\Big] \sim_{s \to +\infty}
 - \| f - \bar{f} \|_{\rm F} \cdot \frac{ rd}{s}.
 \]

 \section{Proof of corollary~\ref{cor:hausd}}
 \label{app:hausd}
  {For $\| \tau\|_\infty \leqslant 2s$, let $\Omega(\tau)$ denote the set of $(\omega,\omega') \in \zb^d\times \zb^d$ such that $\|\omega\|_\infty \leqslant r$, $\| \omega'\|_\infty \leqslant r$, and $\omega-\omega' = \tau$.
 We consider the norm $\Theta$ on the set of Hermitian matrices of dimension $(2r+1)^d$ defined as: 
 \BEAS
 \Theta(\Sigma) = (2r+1)^d \sum_{\| \tau \|_\infty \leqslant 2r} \bigg(
 |\Omega(\tau)|^{-1}\!\!\!\!
 \sum_{(\omega,\omega') \in \Omega(\tau)} \!\!\!\!\!
| \Sigma_{\omega \omega'}|^2 \bigg)^{1/2}.
 \EEAS
 This norm is constructed so that, for a trigonometric polynomial $f$ of degree less than $2r$ represented by a Hermitian matrix $F$, then $\| f \|_{\rm F} = \inf_{Y \in \mathcal{V}_r^\perp} \Theta^\ast( F + Y)$ (taking into account the normalizing factor defining $\varphi$). Thus, 
we have, applying Theorem~\ref{theo:nocond}:
 \BEAS
\inf_{\Sigma' \in \mathcal{K}_s}  \Theta( {\Pi}_{s}^{(r)} \big( \Sigma - \Sigma'))
 & = & \inf_{\Sigma' \in \mathcal{K}_s} \sup_{\Theta^\ast(F) \leqslant 1}
 \tr [ F  {\Pi}_{s}^{(r)} \big( \Sigma - \Sigma')] \\[-.1cm] 
 & = & \sup_{\Theta^\ast(F) \leqslant 1}
  \tr [ F  {\Pi}_{s}^{(r)}   \Sigma ]
  + 
 \inf_{\Sigma' \in \mathcal{K}_s} 
 \tr [- F  {\Pi}_{s}^{(r)}   \Sigma' ]
\\[-.1cm] 
& \leqslant & \! \sup_{\Theta^\ast(F) \leqslant 1}
  \tr [ F  {\Pi}_{s}^{(r)}   \Sigma ]
  + 
 \inf_{\Sigma' \in \widehat{\mathcal{K}}_s} 
 \tr [- F  {\Pi}_{s}^{(r)}   \Sigma' ] \! +\! \Big[  \Big( 1\! -\! \frac{6r^2}{s^2} \Big)^{-d} \!-\! 1 \Big] \\[-.1cm] 
 & \leqslant & \Big[  \Big( 1\! -\! \frac{6r^2}{s^2} \Big)^{-d} \!-\! 1 \Big],
 \EEAS
 by selecting $\Sigma' = \Sigma$ in the bound above. The bound using the Frobenius norm is obtained by computing a lower bound on $\Theta^\ast$ as done in Appendix~\ref{app:techlemma} below (but applying to $r$ instead of $s$.}
 
 \section{Proof of Lemma~\ref{lemma:lemma42}}
 \label{app:techlemma}
 \begin{proof}
{Assuming $f_\ast = 0$ without loss of generality, let $f$ be represented by the Hermitian matrix $F$, and $g$ by the PSD Hermitian matrix $G$, that is, for all $x \in [0,1]^d$, $f(x) = \varphi(x)^\ast H \varphi(x)$ and $g(x) = \varphi(x)^\ast G \varphi(x)$. For $\| \tau\|_\infty \leqslant 2s$, if $\Omega(\tau)$ is the set of $(\omega,\omega') \in \zb^d\times \zb^d$ such that $\|\omega\|_\infty \leqslant s$, $\| \omega'\|_\infty \leqslant s$, and $\omega-\omega' = \tau$, then, using that the space $\mathcal{V}_s$ of Hermitian Toeplitz matrices $H$ is characterized by equal values for $H_{\omega\omega'}$ for $(\omega,\omega') \in \Omega(\tau)$ for each $\tau$,
\BEAS
 \| f - g\|_{\rm F} & = & \!\! \! \sum_{\|\omega\|_\infty \leqslant 2s} \!\!\!\!| \hat{f}(\omega) - \hat{g}(\omega)|  \\[-.15cm]
 & = & \frac{1}{(2s+1)^d}
 \inf_{Y \in \mathcal{V}_s^\perp} \!\! \sum_{\| \tau\|_\infty \leqslant 2s} \!\! \bigg(  | \Omega(\tau)| \!\!\sum_{(\omega,\omega') \in \Omega(\tau)} \!\!\!\!\!
| ( F - G + Y)_{\omega \omega'}|^2 \bigg)^{1/2}
\\
& \geqslant & \inf_{Y \in \mathcal{V}_s^\perp} 
\frac{\min_{\| \tau\|_\infty \leqslant 2s} | \Omega(\tau)|^{1/2}}{(2s+1)^d}
 \bigg( \sum_{\| \tau\|_\infty \leqslant 2s}  \sum_{(\omega,\omega') \in \Omega(\tau)} \!\!\!\!\!
| ( F - G + Y)_{\omega \omega'}|^2 \bigg)^{1/2}
\\
& =  & \frac{\sqrt{2}}{(2s+1)^d}
  \inf_{Y \in \mathcal{V}_s^\perp} 
 \| F - G + Y \|_{\rm Frob} 
 \geqslant \frac{\sqrt{2}}{(2s+1)^d} \inf_{Y \in \mathcal{V}_s^\perp} 
 \| F - G + Y \|_{\rm op}.
\EEAS
We can then take the maximizer above $Y \in  \mathcal{V}_s^\perp$, and we have
\BEAS
c_\ast(f,s) \geqslant \lambda_{\min} ( F + Y)
\geqslant \lambda_{\min} ( G ) - \| F - G + Y \|_{\rm op}
\geqslant \textstyle 0 -\frac{(2s+1)^d}{\sqrt{2}} \varepsilon'(s,f),
\EEAS
where $\| \cdot\|_{\rm Frob}$ denotes the Frobenius norm and   $\| \cdot\|_{\rm op}$  the largest singular value.
}
\end{proof}

\section{Proof of generic lemmas about derivatives}
\label{app:A}
 
 In this appendix, we prove lemmas about derivatives and Fourier decays. 
\subsection{Proof of Lemma~\ref{lemma:derdec}}
\label{app:proofderdec}

 \begin{proof}
{We will show a bound on the Fourier series of  ${g}$} of the form
\BEQ
\label{eq:DK}
| \hat{g}(\omega)| \leqslant  D(k)  \frac{1}{(2+ \| \omega\|_1 )^k},
\EEQ
for a constant $D(k)$ to be determined,
since it implies, for $k\geqslant d + 1$:
\[
\sum_{\| \omega \|_\infty \geqslant s} | \hat{g}(\omega)|   \leqslant    D(k) \sum_{\omega \in \zb^d} \frac{1}{(2+ \| \omega\|_1 )^k} =    D(k) \sum_{t=s}^\infty   \frac{1}{(2+ t )^k} { d + t-1 \choose d-1} ,
\]
by counting the number of $\omega \in \zb^d$ such that $\|\omega\|_1 = t$. This leads to the desired results (in particular by taking $s=0$).

We first start by a simple upper bound on ${ d + t-1 \choose d-1}$, as (using the identity $n^n \leqslant n! e^{n-1}$ applied to $n=d-1$):
\BEAS
{ d + t-1 \choose d-1}
& = & \frac{1}{(d-1)!}(t+1) \cdots (t+d-1) \leqslant \frac{(t+d-1)^{d-1}}{(d-1)!} \\
& \leqslant & 2^{d-2} \frac{t^{d-1}+(d-1)^{d-1}}{(d-1)!} \leqslant \frac{2^{d-2}}{(d-1)!} t^{d-1} + (2e)^{d-2}.
\EEAS
This leads to:
\BEAS
\sum_{\omega \in \zb^d} | \hat{g}(\omega)| & \leqslant & 
D(k) \sum_{t=0}^\infty   \frac{1}{(2+  t )^k}  \Big( \frac{2^{d-2}}{(d-1)!}   t^{d-1} + (2e)^{d-2} \Big)
\\
& \leqslant & 
D(k)  \Big[ \frac{2^{d-2}}{(d-1)!}  \frac{1}{k-d} + (2e)^{d-2} \frac{1}{k-1} \Big] \Big)
 \\
 & \leqslant  & D(k)  \Big[ \frac{2^{d-2}}{(d-1)!}  + \frac{(2e)^{d-2}}{d}   \Big] \Big) \leqslant 2 (2e)^{d-2}    D(k) .
\\
 \sum_{\| \omega \|_\infty \geqslant s} | \hat{g}(\omega)|
 & \leqslant &  D(k) \sum_{t=s}^\infty  \frac{1}{(2+  t )^k}  \Big( \frac{2^{d-2}}{(d-1)!}   t^{d-1} + (2e)^{d-2} \Big)
\\[-.2cm]
& \leqslant &  D(k)  \Big[ \frac{2^{d-2}}{(d-1)!}  \frac{1}{(s+1)^{k-d}} + \frac{(2e)^{d-2}}{d}  \frac{1}{(s+1)^{k-1}}   \Big] \\
& \leqslant   & 
 2 (2e)^{d-2} (s+1)^{d-k} D(k).\EEAS
 
 \paragraph{Proof of \eq{DK}}
To obtain \eq{DK}, we need to be able to bound the product  $| \hat{g}(\omega)| | \omega_j|^{\alpha_1} \cdots | \omega_d|^{\alpha_d}$ for any $\alpha$ such that $\alpha_1+\cdots+ \alpha_d = k$. {For this, we need uniform bounds on all partial derivatives, which we need to obtain from bounds on $
\nabla^k g(x) [ \delta,\dots, \delta ]
$
for all $\delta$ and $k$.
From the polarization Lemma~\ref{lemma:polar}, we have
\BEAS
|\nabla^k g(x) [ \delta_1,\dots,\delta_k]| 
\leqslant \frac{1}{k!}  \Big(\sum_{i=1}^k \| \delta_i\|_1 \Big)^k \cdot \| \nabla^k  g\|_\infty
\leqslant  \frac{1}{k!}  \Big(\sum_{i=1}^k \| \delta_i\|_1 \Big)^k  C \cdot B^k \cdot    k! \cdot \kappa(k),
\EEAS
by definition of $\| \nabla^k g\|_\infty$ and because of the assumptions of the lemma. For any $\alpha$ such that $\alpha_1+\cdots+ \alpha_d = k$, the partial derivative
$\partial_\alpha g(x) = \frac{\partial^k g}{\partial x_1 ^{\alpha_1} \cdots x_d^{\alpha_d}}(x)$ can be written as
$\frac{\partial^k g}{\partial x_{j_1}  \cdots \partial x_{j_k}}(x)$ for $j_1,\dots,j_k \in \{1,\dots,d\}$. Thus, applying the inequality above with~$\delta_i $ the indicator vector of the set $\{ j_i\}$ for each $i \in \{1,\dots,k\}$, we get}
\[
 | \partial_\alpha g(x)| \leqslant |\nabla^k g(x) [ \delta_1,\dots,\delta_k]| 
\leqslant   C \cdot B^k \cdot    k^k \cdot \kappa(k).
\]
Then, by expanding $(2+\|\omega\|_1)^k$ with the multinomial formula, and using the bound
$\hat{g}(\omega) \prod_{i=1}^d |2 \pi \omega_i|^{\alpha_i} 
\leqslant \sup_{x \in [0,1]^d} | \partial_\alpha g(x)|$, we get:
\BEAS
& & | \hat{g}(\omega)| \sum_{\|\alpha\|_1 = k} \frac{k!}{\alpha_0! \alpha_1! \cdots \alpha_d!} 2^{\alpha_0} | \omega_j|^{\alpha_1} \cdots |  \omega_d|^{\alpha_d} \\[-.1cm]
& \leqslant & \sum_{\|\alpha\|_1 = k} \frac{k!}{\alpha_0! \alpha_1! \cdots \alpha_d!} 2^{\alpha_0} C \Big(\frac{B}{2\pi} \Big)^{k-\alpha_0} k^{k- \alpha_0} \kappa(k)  \leqslant   C \Big( 2 + \frac{d Bk}{ 2 \pi} \Big)^k \kappa(k) .\EEAS

This leads to $\ds D(k) \leqslant C  \Big( 2 + \frac{d Bk}{ 2 \pi} \Big)^k \kappa(k)$, and thus the desired result.
\end{proof}

\begin{lemma}[Polarization]
\label{lemma:polar}
Let $u:E^m \to \rb $ be a symmetric $m$-multi-linear form on some normed vector space $E$. Then for all $z_1,\dots,z_m \in E$, we have:  
\[
|u[ z_1,\dots,z_m]|  \leqslant 
\frac{1}{m!}  \Big(\sum_{i=1}^m \| z_i\|_1 \Big)^m \cdot \sup_{\|z\|_1\leqslant 1} u(z,\dots,z).
\]
\end{lemma}
\begin{proof}
We use the polarization identity for the $m$-multilinear form $u: E^m \to E$ and its diagonal $\tilde{u}: z \mapsto u(z,\dots,z)$, see~\cite[Eq.~(A.4)]{thomas2014polarization},
\[
u(z_1,\dots,z_m) = \frac{1}{2^m m!} \sum_{ \varepsilon \in \{0,1\}^m} (-1)^{  \| \varepsilon\|_1} \tilde{u} \Big( \sum_{i=1}^m (-1)^{\varepsilon_i} z_i \Big),
\]
which leads to
\BEAS
| u(z_1,\dots,z_m) | & \leqslant & \frac{1}{2^m m!} \sum_{ \varepsilon \in \{0,1\}^m} \Big(\sum_{i=1}^m \| z_i\|_1 \Big)^m \sup_{\| z \|_1 \leqslant 1} | \tilde{u}(z)|
\\[-.2cm]
& =&  \frac{1}{m!}  \Big(\sum_{i=1}^m \| z_i\|_1 \Big)^m  \sup_{\| z \|_1 \leqslant 1} | \tilde{u}(z)|,
\EEAS
which is the desired result.
\end{proof}

\subsection{Proof of Lemma~\ref{lemma:prod}}
\label{app:proofder}
\begin{proof}
Using Leibniz formula applied to $\varphi_1(t) = h_1(x+t \delta)$,  $\varphi_2(t) = h_2(x+t \delta)$, we have:
\BEAS
(\varphi_1 \varphi_2)^{(m)}(0) & \!\!=\!\! &   \textstyle \sum_{i=0}^m { m \choose i} \varphi_1^{(i)}(0)  \varphi_2^{(m-i)}(0) \\ 
& \!\!\leqslant \!\!& \textstyle C_1 C_2  \|\delta\|_1^m \sum_{i=0}^m { m \choose i} B_1^i B_2^{m-i} i! (m-i)! \kappa_1(i) \kappa_2(m-i) 
\\
&\!\! \leqslant \!\!&  \textstyle  C_1 C_2 \kappa_1(m) \kappa_2(m)   \|\delta\|_1^m m! \sum_{i=0}^m B_1^i B_2^{m-i}
\\
&  \leqslant &  C_1 C_2 \kappa_1(m) \kappa_2(m)   \|\delta\|_1^m (m+1)! \max \{ B_1,B_2\}^m .
\EEAS
\end{proof}

\bibliography{sos_exp}

  \end{document}